\documentclass[a4paper,11pt]{amsart}

\usepackage{a4wide}
\usepackage{pgfplots}
\usepackage{subcaption}
\usepackage[foot]{amsaddr}

\usepackage{xcolor}

\usepackage[breaklinks,bookmarks=false]{hyperref}
\hypersetup{colorlinks, linkcolor=blue, citecolor=blue,urlcolor=blue,plainpages=false}

\newtheorem{theorem}{Theorem}
\newtheorem{lemma}[theorem]{Lemma}

\newtheorem{assumption}[theorem]{Assumption}
\newtheorem{remark}[theorem]{Remark}

\newtheorem{algo}[theorem]{Algorithm}

\numberwithin{theorem}{section}
\numberwithin{equation}{section}

\usepackage{amssymb}
\usepackage{mathrsfs}

\renewcommand{\Re}{\operatorname{Re}}

\newcommand{\eq}{:=}

\newcommand{\pd}[2]{\frac{\partial #1}{\partial #2}}

\newcommand{\grad}{\boldsymbol \nabla}
\renewcommand{\div}{\grad \cdot}
\newcommand{\curl}{\grad \times}

\newcommand{\ccurl}{\boldsymbol{\operatorname{curl}}}
\newcommand{\ddiv}{\operatorname{div}}

\newcommand{\jmp}[1]{\,[\![#1]\!]}

\newcommand{\BA}{\boldsymbol A}
\newcommand{\BB}{\boldsymbol B}

\newcommand{\BE}{\boldsymbol E}

\newcommand{\BH}{\boldsymbol H}
\newcommand{\BI}{\boldsymbol I}
\newcommand{\BJ}{\boldsymbol J}

\newcommand{\BL}{\boldsymbol L}

\newcommand{\BW}{\boldsymbol W}
\newcommand{\BX}{\boldsymbol X}

\newcommand{\ba}{\boldsymbol a}

\newcommand{\bd}{\boldsymbol d}
\newcommand{\be}{\boldsymbol e}

\newcommand{\bj}{\boldsymbol j}

\newcommand{\bn}{\boldsymbol n}
\newcommand{\bo}{\boldsymbol o}
\newcommand{\bp}{\boldsymbol p}

\newcommand{\br}{\boldsymbol r}

\newcommand{\bu}{\boldsymbol u}
\newcommand{\bv}{\boldsymbol v}
\newcommand{\bw}{\boldsymbol w}
\newcommand{\bx}{\boldsymbol x}
\newcommand{\by}{\boldsymbol y}

\newcommand{\CF}{\mathcal F}

\newcommand{\CM}{\mathcal M}
\newcommand{\CN}{\mathcal N}

\newcommand{\CP}{\mathcal P}
\newcommand{\CQ}{\mathcal Q}
\newcommand{\CR}{\mathcal R}
\newcommand{\CS}{\mathcal S}
\newcommand{\CT}{\mathcal T}

\newcommand{\CV}{\mathcal V}

\newcommand{\LC}{\mathscr C}
\newcommand{\LD}{\mathscr D}
\newcommand{\LE}{\mathscr E}

\newcommand{\LP}{\mathscr P}

\newcommand{\LV}{\mathscr V}

\newcommand{\TI}{\textup I}

\newcommand{\BCN}{\boldsymbol{\CN}}

\newcommand{\BCP}{\boldsymbol{\CP}}

\newcommand{\BTI}{\pmb{\TI}}

\newcommand{\SlopeTriangle}[6]
{

    \pgfplotsextra
    {
        \pgfkeysgetvalue{/pgfplots/xmin}{\xmin}
        \pgfkeysgetvalue{/pgfplots/xmax}{\xmax}
        \pgfkeysgetvalue{/pgfplots/ymin}{\ymin}
        \pgfkeysgetvalue{/pgfplots/ymax}{\ymax}

        \pgfmathsetmacro{\xArel}{#1}
        \pgfmathsetmacro{\yArel}{#3}
        \pgfmathsetmacro{\xBrel}{#1-#2}
        \pgfmathsetmacro{\yBrel}{\yArel}
        \pgfmathsetmacro{\xCrel}{\xArel}

        \pgfmathsetmacro{\lnxB}{\xmin*(1-(#1-#2))+\xmax*(#1-#2)} 
        \pgfmathsetmacro{\lnxA}{\xmin*(1-#1)+\xmax*#1} 
        \pgfmathsetmacro{\lnyA}{\ymin*(1-#3)+\ymax*#3} 
        \pgfmathsetmacro{\lnyC}{\lnyA+#4*(\lnxA-\lnxB)}
        \pgfmathsetmacro{\yCrel}{\lnyC-\ymin)/(\ymax-\ymin)} 

        \coordinate (A) at (rel axis cs:\xArel,\yArel);
        \coordinate (B) at (rel axis cs:\xBrel,\yBrel);
        \coordinate (C) at (rel axis cs:\xCrel,\yCrel);

        \draw[#6]   (A)-- node[anchor=north] {#5}
                    (B)--
                    (C)--
                    cycle;
    }
}

\newcommand{\eps}{\varepsilon}

\newcommand{\ee}{{\boldsymbol \eps}}
\newcommand{\mm}{{\boldsymbol \mu}}
\newcommand{\cc}{{\boldsymbol \chi}}

\newcommand{\sig}{{\boldsymbol\sigma}}

\newcommand{\tee}{\widetilde{\ee}}
\newcommand{\tmm}{\widetilde{\mm}}
\newcommand{\tss}{\widetilde{\sig}}

\newcommand{\zero}{\bo}
\newcommand{\norm}[1]{|\!|\!|#1|\!|\!|}

\newcommand{\enorm}[1]{\norm{#1}_{\ccurl,\omega,\Omega}}

\newcommand{\lO}{\ell_\Omega}

\newcommand{\bthe}{\boldsymbol \theta}
\newcommand{\bphi}{\boldsymbol \phi}

\newcommand{\bxi}{\boldsymbol \xi}

\newcommand{\cW}{\LC_{\rm w}}
\newcommand{\cR}{\LC_{\rm r}}
\newcommand{\cI}{\LC_{\rm i}}

\newcommand{\cb}{\LC_{\rm b}}

\newcommand{\cSE}{\gamma_{\rm st}}
\newcommand{\gbaE}{\gamma_{\rm ba}}

\newcommand{\tK}{\widetilde K}
\newcommand{\tF}{\widetilde F}

\newcommand{\eemO}{\varepsilon_{\Omega,\min}}
\newcommand{\eeMO}{\varepsilon_{\Omega,\max}}

\newcommand{\mmmO}{\mu_{\Omega,\min}}
\newcommand{\mmMO}{\mu_{\Omega,\max}}

\newcommand{\eemK}{\varepsilon_{\tK,\min}}
\newcommand{\mmMK}{\mu_{\tK,\max}}

\newcommand{\essinf}[1]{\underset{\substack{#1}}{\operatorname{ess} \operatorname{inf}}\;}
\newcommand{\esssup}[1]{\underset{\substack{#1}}{\operatorname{ess} \operatorname{sup}}\;}

\newcommand{\pol}{\mathcal{P}}

\newcommand{\osc}{\operatorname{osc}}

\newcommand{\tnorm}[1]{\norm{#1}}

\title%
[Frequency-explicit a posteriori error estimates for Maxwell's equations]%
{Frequency-explicit a posteriori error estimates for finite element discretizations of Maxwell's equations}

\author{T. Chaumont-Frelet$^{\star,\dagger}$}
\author{P. Vega$^{\star,\dagger}$}

\address{\vspace{-.5cm}}
\address{\noindent \tiny \textup{$^\star$Inria, 2004 Route des Lucioles, 06902 Valbonne, France}}
\address{\noindent \tiny \textup{$^\dagger$Laboratoire J.A. Dieudonn\'e, Parc Valrose, 28 Avenue Valrose, 06108 Nice Cedex 02, 06000 Nice, France}}

\begin{document}

\begin{abstract}
We consider residual-based {\it a posteriori} error estimators for Galerkin 
discretizations of time-harmonic Maxwell's equations.
We focus on configurations where the frequency is high,
or close to a resonance frequency, and derive reliability and efficiency estimates.
In contrast to previous related works, our estimates are frequency-explicit.
In particular, our key contribution is to show that even if the constants appearing
in the reliability and efficiency estimates may blow up on coarse meshes, they become
independent of the frequency for sufficiently refined meshes. Such results were previously known
for the Helmholtz equation describing scalar wave propagation problems, and we show
that they naturally extend, at the price of many technicalities in the proofs,
to Maxwell's equations. Our mathematical analysis is performed in the 3D case
and covers conforming N\'ed\'elec discretizations of the first and second family.
We also present numerical experiments in the 2D case, where Maxwell's equations are
discretized with N\'ed\'elec elements of the first family. These illustrating examples
perfectly fit our key theoretical findings and suggest that our estimates are sharp.

\vspace{.5cm}
\end{abstract}

\maketitle

\section{Introduction}

Maxwell's equations constitute the central model of electrodynamics \cite{griffiths_1999a}.
They are ubiquitously employed to describe the propagation of electromagnetic fields, and
encompass a wide range of applications, including radar imaging \cite{dorf_2006a},
telecommunications \cite{russer_2006a}, and nanophotonics \cite{gaponenko_2010a},
just to cite a few. In realistic geometries,
analytical solutions to Maxwell's equations are out of reach, which motivates the
development of numerical schemes to compute approximate solutions. While several approaches,
such as finite differences \cite{yee_1966a} or boundary elements \cite{bendali_1984a},
are available for the problem under consideration, we focus on finite element methods
in this work \cite{monk_2003a}. The latter are especially suited in the case of heterogeneous
media with complex geometries due to the flexibility of unstructured meshes.

Current computer hardware enables the realization of three-dimensional simulations,
which are of practical interest, but the associate computational costs are still
important regarding power consumption and simulation time. As a result, not only do
numerical schemes need to be accurate and robust, but they also have to be
as efficient as possible. In the context of finite element methods, an attractive
idea to limit the computational cost is to adapt the mesh size and/or the polynomial
degree of the basis functions only locally in the areas of the domain where the solution
exhibits a complicated behavior.

These local refinements may be carried out by an iterative refinement process that
is driven by {\it a posteriori} error estimators. Besides, error estimators can also
be employed to quantitatively estimate the discretization error, which enable practitioners
to decide whether the numerical solution is sufficiently accurate for the application purposes.
As a result, {\it a posteriori} error estimators have attracted an increasing
attention over the last decades. We refer the reader to
\cite{%
ainsworth_oden_2000a,%
demkowicz_2006a,%
verfurth_1994},
and the references therein.

Early works on {\it a posteriori} error estimation have focused on scalar coercive
problems \cite{ainsworth_oden_2000a,verfurth_1994}. It turns out that there are two
key properties an {\it a posteriori} estimator should satisfy. On the one hand, ``reliability''
states that the estimator is an upper bound to the discretization error, i.e., there exists
a constant
$C_{\rm rel}$ such that
\begin{subequations}
\label{eq_intro_rel_eff}
\begin{equation}
\label{eq_intro_rel}
\tnorm{e}_\Omega \leq C_{\rm rel} \eta
\end{equation}
where $\tnorm{e}_\Omega$ is a measure of the error in the whole computational
domain. On the other hand, ``efficiency'' refers to the fact that the estimator associated with
each element $K$ of the mesh is a lower bound for the error measured in a small
region around $K$. Specifically, there exists a constant $C_{\rm eff}$ such that
\begin{equation}
\label{eq_intro_eff}
\eta_K \leq C_{\rm eff} \tnorm{e}_{\tK} + \osc_{\CT_{K,h}},
\end{equation}
\end{subequations}
where $\tnorm{e}_{\tK}$ is a measure of the error in a small area $\tK \supset K$,
and $\osc_{\CT_{K,h}}$ is a ``data oscillation'' term linked to the right-hand side
(see Section \ref{section_oscillation}).

In the context of scalar elliptic problems with constant coefficients,
$C_{\rm rel}$ and $C_{\rm eff}$ only depend on the ``shape-regularity''
parameter of the discretization mesh (see Section \ref{section_mesh} below).
These results were subsequently extended to scalar time-harmonic wave propagation problems
\cite{chaumontfrelet_ern_vohralik_2021a,dorfler_sauter_2013a,sauter_zech_2015a}. It is now
well-known that in contrast to the elliptic case, the constants $C_{\rm rel}$ and $C_{\rm eff}$
not only depend on the shape-regularity parameter, but also on the frequency. Specifically,
for a fixed mesh, the estimates in \eqref{eq_intro_rel_eff} deteriorate when the
frequency increases.

Here, we focus on time-harmonic Maxwell's equations set in a Lipschitz polyhedral domain
$\Omega \subset \mathbb R^3$. Namely, for a fixed frequency $\omega > 0$, and a
given current density $\BJ: \Omega \to \mathbb C^3$, the (unknown) electric field
$\BE: \Omega \to \mathbb C^3$ satisfies
\begin{equation}
\label{eq_maxwell_strong}
\left \{
\begin{array}{rcll}
-\omega^2 \ee \BE + \curl (\mm^{-1} \curl \BE) &=& i\omega\BJ & \text{ in } \Omega,
\\
\BE \times \bn &=& \zero & \text{ on } \partial \Omega,
\end{array}
\right .
\end{equation}
where $\partial \Omega$ is the boundary of $\Omega$, and
$\bn$ is the outward normal unit vector to $\partial \Omega$.
As we explain in detail in Section \ref{section_coefficients},
the complex tensor-valued coefficients $\ee$ and $\mm$
represent the electromagnetic properties of the materials
contained in $\Omega$. Notice than since $\ee$ and $\mm$
are allowed to take complex values, our model problem
\eqref{eq_maxwell_strong} encompasses a large number of scenarios,
including the presence of conductive materials, as well as perfectly
matched layers \cite{berenger_1994,berenger_1996,monk_2003a}, that
approximate a radiation condition in the context of scattering problems
(see Remark \ref{remark_pml}). We have chosen the boundary condition
$\BE \times \bn = \bo$ to fix the ideas, and the other natural condition, namely
$(\curl \BE) \times \bn = \bo$, can be covered without any difficulty. However, impedance
boundary conditions, $(\curl \BE) \times \bn + \lambda (\BE \times \bn) \times \bn = \bo$
($\lambda \in \mathbb C$), would require substantial dedicated work.

Several works deal with {\it a posteriori} error estimation for
Galerkin discretizations of simpler versions of \eqref{eq_maxwell_strong},
including the magnetostatic equations (where $\ee \eq \bo$)
\cite{braess_schoberl_2008a,chen_qiu_shi_2018}, and the case where either
$\ee$ is negative-definite or has a definite imaginary part
\cite{beck_hiptmair_hoppe_wohlmuth_2000a,cochezdhondt_nicaise_2007,%
nicaise_creuse_2003a,schoberl_2008}. In both scenarios, the sesquilinear form associated
with the problem is coercive, which greatly simplifies the analysis. We are actually only
aware of a handful of works where reliability and efficiency are established for indefinite
time-harmonic Maxwell's equations \cite{chen_wang_zheng_2007,monk_1998}. In
\cite{chen_wang_zheng_2007,monk_1998} however, the efficiency and reliability constants implicitly
depend on the frequency, even asymptotically. The convergence analysis
of adaptive schemes for \eqref{eq_maxwell_strong} is presented in
\cite{he_yang_wang_2020,zhong_chen_shu_wittum_xu_2012}. There again, the dependence on the
frequency is not explicitly analyzed, and as a consequence, ``sufficiently refined'' meshes
are required to start the adaptive process. As shown, for instance in
\cite{chaumontfrelet_ern_vohralik_2021a}, the reliability
constant may be surprisingly large in some cases, leading to important underestimation of
the error (we also refer the reader to the numerical experiments presented in Section
\ref{sec_numerics}). This is especially problematic when the estimator is
used as a stopping criterion.
On the other hand, the convergence speed of adaptive schemes also depends on these constants.
As a result, it is of interest to estimate the size of these constants and to identify mesh
sizes for which the estimator can be trusted.

In this work, we analyze the efficiency and reliability of residual estimators for
time-harmonic Maxwell's equations \eqref{eq_maxwell_strong}
discretized with N\'ed\'elec finite elements. In particular,
we provide a thorough inspection of the behavior of the efficiency
and reliability constants when the frequency is large and/or close to a resonance frequency.
Our key contributions are twofold. On the one hand, we show that the efficiency
constant is bounded independently of the frequency as soon as the number of
degrees of freedom (dofs) per wavelength is bounded below. On the other hand,
we establish that the reliability constant may be bounded independently of the frequency,
assuming the mesh is sufficiently refined, or the polynomial degree is sufficiently large.
Specifically, our findings are summarized by the estimates
\begin{subequations}
\label{eq_intro_rel_eff_maxwell}
\begin{equation}
\label{eq_intro_rel_maxwell}
\omega \|\BE-\BE_h\|_{\ee,\Omega}
+
\|\curl(\BE-\BE_h)\|_{\mm^{-1},\Omega}
\leq
C \left (1 + \gbaE \right ) \eta,
\end{equation}
and
\begin{equation}
\label{eq_intro_eff_maxwell}
\eta_K
\leq
C p^{3/2} \left (1 + \frac{\omega h_K}{pc_{\tK,\min}}\right )
\left (
\omega \|\BE-\BE_h\|_{\ee,\tK}
+
\|\curl(\BE-\BE_h)\|_{\mm^{-1},\tK}
\right )
+
\osc_{\CT_{K,h}},
\end{equation}
where $h_K$ is the diameter of the element $K$, $p$ is the polynomial degree,
$\BE$ and $\BE_h$ are the solution to \eqref{eq_maxwell_strong} and its discrete approximation,
$c_{\tK,\min}$ is the minimum wavespeed in the patch around $K$,
and $\gbaE$ is the so-called ``approximation factor'' that measures the approximation properties
of the finite element space, and that is properly introduced in Section
\ref{section_approximation_factor} below.
We refer the reader to Theorems \ref{rel_IPDG} and \ref{eff_IPDG}.
\end{subequations}

In \eqref{eq_intro_rel_eff_maxwell}, the frequency appears in the reliability
and efficiency constants through the terms $\gbaE$ and $\omega h_K/(pc_{\tK,\min})$.
On the one hand, the approximation factor $\gbaE $ is now standardly used in the
{\it a posteriori} error analysis of the Helmholtz equation, and we refer the reader to
\cite{chaumontfrelet_ern_vohralik_2021a,dorfler_sauter_2013a,sauter_zech_2015a}
where the notation $\sigma$ stands for the approximation factor, as well as in the
{\it a priori} error analysis of Helmholtz problems
\cite{chaumontfrelet_nicaise_2019a,melenk_sauter_2011a}
and Maxwell's equations \cite{melenk_sauter_2020a,nicaise_tomezyk_2020},
where the symbol $\eta$ is employed. In general, it is complicated
to derive fully-explicit estimates for the approximation factor
\cite{chaumontfrelet_ern_vohralik_2021a}. However, at least
in the case of scalar wave propagation problems, qualitative upper
bounds are available for several configurations of interest
\cite{chaumontfrelet_nicaise_2018a,chaumontfrelet_nicaise_2019a,melenk_sauter_2011a}.
The analysis of the approximation factor for Maxwell's equations is
very recent, so that less results are currently available
\cite{chaumontfrelet_vega_2021a,melenk_sauter_2020a,nicaise_tomezyk_2020}.

On the other hand, since $\lambda_K \eq c_{\tK,\min}/(2\pi\omega)$
denotes the minimal wavelength in a neighborhood of $K$,
$(\omega h_K/(pc_{\tK,\min}))^{-1} \simeq \lambda_K(h_K/p)^{-1}$
is a measure of the number of dofs per wavelength, locally around $K$.
In particular, the condition $\omega h_K/(p c_{\tK,\min}) \leq C$
means that there are ``sufficiently many'' dofs per wavelength.
It is thus natural for this term to appear in \eqref{eq_intro_rel_eff_maxwell}.
In addition, since it is required to have a sufficient number of
dofs per wavelength for approximability reasons, it is expected that
$\omega h_K/(p c_{\tK,\min})$ is small in scenarios of interest.

Our main findings are almost identical to recently established results concerning
scalar wave propagation problems modeled by the Helmholtz equation
\cite{chaumontfrelet_ern_vohralik_2021a,dorfler_sauter_2013a,sauter_zech_2015a}.
As a result, they appear as a natural extension of now well-established results
for scalar wave propagation problems to Maxwell's equations.
However, the proofs are largely complicated by the functional framework
based on the $\BH(\ccurl)$ Sobolev space.
In short, the solution to Maxwell's equations is only smooth when the right-hand
side is $\BH(\ddiv)$-conforming. On the other hand, N\'ed\'elec finite element spaces
are $\BH(\ccurl)$-conforming, but \emph{not} $\BH(\ddiv)$-conforming.
This fact raises several complications when using the error as a right-hand side in
duality arguments, which is common in the analysis of high-frequency wave propagation
problems. Over the past decade, workarounds have been developed in the context of
{\it a priori} analysis
\cite{chaumontfrelet_2019a,%
chaumontfrelet_nicaise_pardo_2018a,%
ern_guermond_2018a,%
zhong_shu_wittum_xu_2009a},
but this work is, to the best of our knowledge, the first contribution to use them in the
context of {\it a posteriori} analysis.

The remainder of this work is organized as follows. Section \ref{sec_settings}
presents the model problem we consider and gathers notations as well as key preliminary
results. In Section \ref{sec_IPDG}, we derive our main results, leading to
\eqref{eq_intro_rel_eff_maxwell}. We report numerical experiments that illustrate our key
findings in Section \ref{sec_numerics}, and draw our conclusions in Section \ref{sec_conclusion}.

\section{Settings}
\label{sec_settings}

Here, we review key notations and preliminary results.

\subsection{Domain and coefficients}
\label{section_coefficients}

We consider Maxwell's equations \eqref{eq_maxwell_strong}
in a Lipschitz polyhedral domain $\Omega \subset \mathbb R^3$.
We do \emph{not} assume that $\Omega$ is simply connected, and
denote by $\lO \eq \sup_{\bx,\by \in \Omega} |\bx-\by|$ the diameter of $\Omega$.

The electromagnetic properties of the materials contained inside $\Omega$
are described by two symmetric tensor-valued functions $\ee,\mm: \Omega \to \CS(\mathbb C^3)$.
For the sake of simplicity, we assume that we can partition
$\Omega$ into a set $\LP$ of non-overlapping polyhedral subdomains $P$ such that
$\ee|_P$ and $\mm|_P$ are constant for all $P \in \LP$.
The short-hand notation $\cc \eq \mm^{-1}$ will also be useful.
Notice that the above tensor fields are symmetric, but \emph{not} self-adjoint.

If $\bphi \in \{\ee,\mm,\cc\}$ is any of the aforementioned tensor fields,
we introduce the notations
\begin{equation*}
\phi_{\min}(\bx)
\eq
\min_{\substack{\bu \in \mathbb C^{3} \\ |\bu| = 1}} \Re \bphi(\bx) \bu \cdot \overline{\bu},
\qquad
\phi_{\max}(\bx)
\eq
\max_{\substack{\bu \in \mathbb C^{3} \\ |\bu| = 1}}
\max_{\substack{\bv \in \mathbb C^{3} \\ |\bv| = 1}}
\Re \bphi(\bx) \bu \cdot \overline{\bv},
\end{equation*}
as well as $\phi_{D,\min} \eq \essinf{\bx \in D} \phi_{\min}(\bx)$
and $\phi_{D,\max} \eq \esssup{\bx \in D} \phi_{\max}(\bx)$
for any open set $D \subset \Omega$. Additionally, we assume that $\phi_{\Omega,\min} > 0$.

Similarly, we will employ the notations
$c_{D,\min} \eq \sqrt{\mu_{D,\min}/\varepsilon_{D,\max}}$
and
$c_{D,\max} \eq \sqrt{\mu_{D,\max}/\varepsilon_{D,\min}}$
for the minimum and maximum wavespeed in $D$.

\begin{remark}[Conductive materials and perfectly matched layers]
\label{remark_pml}
The coefficients $\ee$ and $\mm$ are usually meant to represent the
electric permittivity and the magnetic permeability of the materials
contained inside the computational domain $\Omega$
(see, e.g., \S4.4 and \S6.4 of \cite{griffiths_1999a}). In this case,
these tensors are real-valued. As mathematical convenience, we allow complex-valued
tensors to treat additional physical effects in a unified framework.

Since the frequency is fixed, we can for instance handle a material with (real-valued)
permittivity $\tee$, permeability $\tmm$ and conductivity $\tss$,
by setting $\ee \eq \tee - (1/i\omega) \tss$ and $\mm \eq \tmm$,
as done in \cite{chaumontfrelet_nicaise_pardo_2018a} for instance.
One readily sees that the proposed framework covers this scenario.

Another situation of importance is the case of perfectly matched layers (PML),
that are widely employed to mimic the effect of Silver-M\"uller's radiation
condition in unbounded media. Considering for the sake of simplicity the
Cartesian PML approach \cite{berenger_1994,berenger_1996,monk_2003a}, the
physical coefficients $\tee$ and $\tmm$ are modified as follows. Assuming
the region of interest $\Omega_0$ is contained in the cube $(-L,L)^3$,
the Cartesian PML approach consists in selecting a largest computational domain
$\Omega_0 \subset (-L,L)^3 \subset \Omega$, and for $\bx \in \Omega$ defining
\begin{equation*}
\ee \eq \BB^{-1}\widetilde{\ee}\BA^{-1}
\qquad\text{and}\qquad
\mm\eq\BB^{-1}\widetilde{\mm}\BA^{-1},
\end{equation*}
where $\BA$ and $\BB$ are defined as
\begin{align*}
\BA \eq \left (
\begin{array}{ccc}
\frac{1}{d_{2}d_{3}} & 0                    & 0                    \\ 
0                    & \frac{1}{d_{1}d_{3}} & 0                    \\ 
0                    & 0                    & \frac{1}{d_{1}d_{2}}
\end{array}
\right )
\qquad\text{and}\qquad
\BB \eq \left(
\begin{array}{ccc}
d_1 & 0   & 0 \\ 
0   & d_2 & 0 \\ 
0   & 0 & d_3
\end{array}
\right),
\end{align*}
with, for each space dimension $1 \leq j \leq 3$, $d_j(\bx) \eq 1 - \sigma(\bx_j)/(i\omega)$,
$\sigma(t) = 0$ for $t \in (-L,L)$ and $\sigma(t) = \sigma_\star > 0$ if $|t| > L$.
Assuming for the sake of simplicity that the physical coefficients are
transverse isotropic (i.e. the tensors are diagonal), straightforward computations reveal that
\begin{equation*}
\widetilde \varepsilon_{\Omega,\min} \geq \left(1-\frac{\sigma_{\star}^2}{\omega^2}\right) \varepsilon_{\Omega,\min}
\qquad\text{and}\qquad
\widetilde \mu_{\Omega,\min} \geq \left(1-\frac{\sigma_{\star}^2}{\omega^2}\right) \mu_{\Omega,\min},
\end{equation*}
so that this approach fits to our theoretical framework if $\sigma_{\star} < \omega$.
\end{remark}

\subsection{Functional spaces}
\label{functional_spaces}

In the following, if $D \subset \Omega$,
$L^2(D)$ denotes the space of square-integrable
complex-valued function over $D$, see e.g. \cite{adams_fournier_2003a},
and $\BL^2(D) \eq \left (L^2(D) \right )^3$. We equip $\BL^2(D)$ with the
following (equivalent) norms
\begin{equation*}
\|\bw\|_D^2 \eq \int_D |\bw|^2,
\qquad
\|\bw\|_{\bphi,D}^2 \eq \Re \int_D \bphi \bw \cdot \overline{\bw},
\qquad
\bw \in \BL^2(D)
\end{equation*}
with $\bphi \eq \ee$ or $\cc$. We denote by $(\cdot,\cdot)_D$ the inner-product of $\BL^2(D)$,
and we drop the subscript when $D = \Omega$.
If $F \subset \overline{\Omega}$ is a two-dimensional measurable planar subset,
$\|\cdot\|_F$ and $\langle \cdot,\cdot \rangle_F$
respectively denote the natural norm and inner-product of
both $L^2(F)$ and $\BL^2(F) \eq \left (L^2(F)\right )^3$.

Classically \cite{adams_fournier_2003a}, we employ the notation $H^1(D)$
for the usual Sobolev space of functions $w \in L^2(D)$ such that
$\grad w \in \BL^2(D)$. We also set $\BH^1(D) \eq \left (H^1(D)\right )^3$
and introduce the semi-norms
\begin{equation*}
\|\grad \bw\|_D^2
\eq
\sum_{j,k = 1}^3\int_D \left |\pd{\bw_j}{\bx_k}\right |^2,
\qquad
\|\grad \bw\|_{\bphi,D}^2
\eq
\sum_{j,k = 1}^3 \int_D \phi_{\max} \left |\pd{\bw_j}{\bx_k}\right |^2,
\end{equation*}
for $\bw \in \BH^1(\Omega)$ and $\bphi \eq \ee$ or $\cc$.

We shall also need Sobolev spaces of vector-valued functions with well-defined
rotation and divergence \cite{girault_raviart_1986a}. Specifically, we denote
by $\BH(\ccurl,D)$ the space of functions $\bw \in \BL^2(D)$
with $\curl \bw \in \BL^2(D)$, that we equip with the ``energy'' norm
\begin{equation*}
\norm{\bv}_{\ccurl,\omega,D}^2
\eq
\omega^2 \|\bv\|_{\ee,D}^2 + \|\curl \bv\|_{\cc,D}^2,
\quad \bv \in \BH(\ccurl,D).
\end{equation*}

In addition, if $\bxi: D \to \mathbb C^3$ is a measurable tensor-valued
function we will use the notation $\BH(\ddiv,\bxi,D)$ for the
set of functions $\bw \in \BL^2(D)$ with $\div (\bxi \bw) \in L^2(D)$,
and we will write $\BH(\ddiv^0,\bxi,D)$ for the set of fields
$\bw \in \BH(\ddiv,\bxi,D)$ such that $\div (\bxi \bw) = 0$
in $D$. When $\bxi = \BTI_3$, the identity tensor,
we simply write $\BH(\ddiv,D)$ and $\BH(\ddiv^0,D)$.

For any of the aforementioned spaces $\LV$, the notation
$\LV_0$ denotes the closure of smooth, compactly supported, functions
into $L^2(D)$ (or $\BL^2(D))$ with respect to the norm
$\LV$. These spaces also correspond to the kernel of the naturally associated
trace operators \cite{adams_fournier_2003a,girault_raviart_1986a}.

Finally, if $\LD$ is a collection of disjoint sets $D \subset \Omega$
and $\LV(D)$ is any of the aforementioned spaces, $\LV(\LD)$ stands for the
``broken'' space of functions $\bv \in \BL^2(\Omega)$ (or $L^2(\Omega)$)
such that $\bv|_D \in \LV(D)$ for all $D \in \LD$. We employ the same
notation for the inner-products, norms and semi-norms of $\LV(\LD)$ and $\LV(D)$,
with the subscript $\LD$ instead of $D$.

\subsection{Variational formulation}

In the remaining of this work, we assume that $\BJ \in \BH(\ddiv,\Omega)$. Then,
we may recast \eqref{eq_maxwell_strong} into a weak formulation, which consists in looking for
$\BE \in \BH_0(\ccurl,\Omega)$ such that
\begin{equation}
\label{eq_maxwell_weak}
b(\BE,\bv) = i\omega(\BJ,\bv) \qquad \forall \bv \in \BH_0(\ccurl,\Omega)
\end{equation}
where
$b(\be,\bv) \eq -\omega^2 (\ee\be,\bv) + (\cc\curl\be,\curl\bv)$
for all $\be,\bv \in \BH_0(\ccurl,\Omega)$.
It is easily seen that the sesquilinear form $b(\cdot,\cdot)$
satisfies the following ``G\aa rding inequality''
\begin{equation}
\label{eq_garding_inequality}
\enorm{\be}^2 = \Re b(\be,\be) + 2\omega^2\|\be\|_{\ee,\Omega}^2
\qquad
\forall \be \in \BH_0(\ccurl,\Omega).
\end{equation}

\subsection{Computational mesh}
\label{section_mesh}

We consider a mesh $\CT_h$ that partitions $\Omega$ into non-overlapping
tetrahedral elements $K$. We assume that the mesh $\CT_h$ is conforming
in the sense of \cite{ciarlet_2002a}, which means that the intersection
$\overline{K_+} \cap \overline{K_-}$ of two distinct elements
$K_\pm \in \CT_h$ is either empty, or a single vertex, edge, or face
of both $K_-$ and $K_+$. We further denote by $\CF_h^{\rm e}$ the set of exterior
faces lying on the boundary $\partial \Omega$ and by $\CF_h^{\rm i}$ the remaining
(interior) faces.

We also require that the mesh $\CT_h$ is conforming with the physical
partition $\LP$. Specifically, we assume that for each $K \in \CT_h$, there exists
$P \in \LP$ such that $K \subset P$. It equivalently
means that the coefficients are constant over each element $K \in \CT_h$, and
that the interfaces of the partition $\LP$ are covered by mesh faces
$F \in \CF_h \eq \CF_h^{\rm e} \cup \CF_h^{\rm i}$.

For $K \in \CT_h$, $\CF_K \subset \CF_h$ denotes the faces of $K$, and the notations
\begin{equation*}
h_K \eq \sup_{\bx,\by \in K} |\bx-\by|,
\qquad
\rho_K \eq \sup \left \{ r > 0 \; | \; \exists \bx \in K: \; B(\bx,r) \subset K \right \},
\end{equation*}
stand for the diameter of $K$ and the radius of the largest ball contained in $\overline{K}$.
We write $\beta_K \eq h_K/\rho_K$ for the ``shape-regularity'' parameter of $K$, and
$\beta \eq \max_{K \in \CT_h} \beta_K$ for the global shape-regularity parameter of
the mesh. The (global) mesh size is defined as $h \eq \max_{K \in \CT_h} h_K$.

We define the jump of $\bv \in \BH^1(\CT_h)$ through
$F \eq \partial K_- \cap \partial K_+ \in \CF_h^{\rm i}$ by
\begin{equation*}
\jmp{\bv}|_F \eq \bv_+ (\bn_+ \cdot \bn_F) + \bv_- (\bn_- \cdot \bn_F),
\end{equation*}
where $\bv_\pm$ is the trace of $\bv$ on $F$ from the interior of $K_{\pm}$
and $\bn_\pm$ is the unit normal pointing outward of $K_\pm$. For exterior faces
$F \in \CF_h^{\rm e}$, we simply set $\jmp{\bv}|_F \eq \bv|_F$.

If $K \in \CT_h$ and $F \in \CF_h$, then
\begin{equation*}
\CT_{K,h}
\eq
\left \{
K' \in \CT_h \; | \; \overline{K} \cap \overline{K'} \neq \emptyset
\right \},
\quad
\CT_{F,h}
\eq
\left \{
K' \in \CT_h \; | \;
F \subset \partial K'
\right \},
\end{equation*}
denote the mesh patches associated with $K$ and $F$, and we use the symbols
$\tK \eq \operatorname{Int} ( \cup_{K' \in \CT_{K,h}} \overline{K'} )$
and $\tF \eq \operatorname{Int} ( \cup_{K \in \CT_{F,h}} \overline{K} )$,
for the associated open domains.

Finally, if $\CT \subset \CT_h$ and $\CF \subset \CF_h$ are collections of elements and faces,
we write
\begin{equation*}
(\cdot,\cdot)_\CT
\eq
\sum_{K \in \CT} (\cdot,\cdot)_K,
\quad
\langle \cdot,\cdot \rangle_{\partial \CT}
\eq
\sum_{K \in \CT} \langle \cdot,\cdot \rangle_{\partial K},
\quad
\langle \cdot,\cdot \rangle_{\CF}
\eq
\sum_{F \in \CF} \langle \cdot,\cdot \rangle_{F}.
\end{equation*}

\subsection{Polynomial spaces}

In the following, for all $K \in \CT_h$ and for $q \geq 0$,
$\CP_q(K)$ stands for the space of
(complex-valued) polynomials defined over $K$ and
$\BCP_q(K) \eq \left (\CP_q(K)\right )^3$. 
We shall also require the N\'ed\'elec polynomial space
$\BCN_q(K) \eq \BCP_q(K) + \bx \times \BCP_q(K)$.
The inclusions $\grad \left (\CP_{q+1}(K)\right ) \subset \BCN_q(K) \subset \BCP_{q+1}(K)$
hold true.

If $\CT \subset \CT_h$, $\CP_q(\CT)$, $\BCP_q(\CT)$ and $\BCN_q(\CT)$ respectively
stand for the space of functions that are piecewise in $\CP_q(K)$, $\BCP_q(K)$ and $\BCN_q(K)$
for all $K \in \CT$.

In the remaining, we consider a fixed polynomial degree $p \geq 1$. Then,
\begin{equation*}
V_h \eq \CP_{p}(\CT_h) \cap H^1_0(\Omega), \qquad
\BW_h \eq \BCN_{p-1}(\CT_h) \cap \BH_0(\ccurl,\Omega)
\end{equation*}
are the usual Lagrange and N\'ed\'elec finite element spaces \cite{monk_2003a}.
Remark that we have $\grad V_h \subset \BW_h$.
We also recall that $\BCP_p(\CT_h) \cap \BH_0(\ccurl,\Omega)$
is the second family of N\'ed\'elec spaces \cite{nedelec_1986a}.

\subsection{Quasi-interpolation}
\label{section_quasi_interpolation}
There exists two operators $\CQ_h: H^1_0(\Omega) \to V_h$ and
$\CR_h: \BH^1_0(\Omega) \to \BW_h$ and a constant $\cI$ that
only depends on $\beta$ such that for each $K \in \CT_h$
\begin{equation}
\label{eq_quasi_interpolation_H1}
\frac{p}{h_K}\|w-\CQ_hw\|_K + \sqrt{\frac{p}{h_K}}\|w-\CQ_hw\|_{\partial K}
\leq
\cI
\|\grad w\|_{\tK}
\end{equation}
for all $w \in H^1_0(\Omega)$ and
\begin{equation}
\label{eq_quasi_interpolation_Hc}
\frac{p}{h_K}\|\bw-\CR_h\bw\|_K
+
\sqrt{\frac{p}{h_K}}\|(\bw-\CR_h\bw) \times \bn\|_{\partial K}
\leq
\cI
\|\grad \bw\|_{\widetilde K}
\end{equation}
for all $\bw \in \BH^1_0(\Omega)$.
A construction of $\CQ_h$ can be found in \cite{hiptmair_pechstein_2019,melenk_2005a}.
For $\CR_h$, we employ the interpolation operator introduced in
\cite{ern_guermond_2017a} when $p=1$. On the other hand, if $p \geq 2$, $\CR_h$ is
defined by employing $\CQ_h$ with degree $p-1$ componentwise, which
is possible since $\left (\CP_{p-1}(\CT_h) \cap H^1_0(\Omega)\right )^3 \subset \BW_h$.

\subsection{Bubble functions, inverse inequalities and extension operator}
\label{bubbles}
\begin{subequations}
Following \cite{dorfler_sauter_2013a,melenk_wohlmuth_2001a}, we introduce,
for each element $K \in \CT_h$ and each face $F \in \CF_h$, bubble functions
$b_K \in C^0(\overline{K})$ and $b_F \in C^0(\overline{F})$, with
$0 \leq b_K, b_F \leq \cb$, such that%
\footnote{The results in \cite{melenk_wohlmuth_2001a} are rigorously stated
for the two-dimensional case. However, as observed in \cite[Theorem 4.12]{dorfler_sauter_2013a},
these results naturally extend to the three-dimensional case.}
\begin{equation}
\label{eq_norm_bubble}
\|w\|_K \leq \cb p\|b_K^{1/2} w\|_K, \qquad \|v\|_F \leq \cb p \|b_F^{1/2} v\|_F
\end{equation}
for all $w \in \CP_p(K)$ and $v \in \CP_p(F)$ (see \cite[Theorem 2.5]{melenk_wohlmuth_2001a}).
Here, $\cb>0$ is a constant depending on shape regularity parameter $\beta$.
Besides, combining the estimate in \cite[Theorem 2.5]{melenk_wohlmuth_2001a}
with the product rule shows that
\begin{equation}
\label{eq_inv_bubble}
\|\grad(b_K w)\|_K \leq \cb \frac{p}{h_K}\|b_K^{1/2}w\|_K
\qquad
\forall w \in \CP_p(K).
\end{equation}
Finally, \cite[Lemma 2.6]{melenk_wohlmuth_2001a} guarantees the existence of
an extension operator $\LE: \pol_p(F) \to H^1_0(\tF)$ such that $\LE(v)|_F= b_F v$ and
\begin{equation}
\label{eq_ext_bubble}
ph_F^{-1/2}\|\LE(v)\|_{\CT_{F,h}}
+
p^{-1}h_{F}^{1/2}\|\grad \LE(v)\|_{\CT_{F,h}}
\leq
\cb \|v\|_F \qquad \forall v \in \CP_p(F).
\end{equation}
\end{subequations}
Identical results hold for vector-valued functions, applying the above estimates componentwise.

\subsection{Data oscillation}
\label{section_oscillation}

Classically, our {\it a posteriori} error estimates include ``data-oscillation'' terms.
These terms include a ``projected source term'' $\BJ_h \in \BCP_p(\CT_h)$,
that is defined elementwise, for each $K \in \CT_h$, as the unique polynomial
$\BJ_h \in \BCP_p(K)$ such that
\begin{equation*}
\frac{\omega^2 h_K^2}{p^2c_{\tK,\min}^2}(\BJ_h,\bv_h)_K
+
\frac{h_K^2}{p^2}(\div\BJ_h,\div\bv_h)_K
=
\frac{\omega^{2} h_K^{2}}{p^2c_{\tK,\min}^{2}}(\BJ,\bv_h)_K
+
\frac{h_K^2}{p^2}(\div\BJ,\div\bv_h)_K
\end{equation*}
for all $\bv_h \in \BCP_p(K)$.
Notice that while other ``traditional'' projections or quasi-interpolation operators
could be employed, we select this particular definition as it minimizes the terms
appearing in the final error bounds.

We are now ready to introduce the elementwise oscillation terms%
\begin{align*}
\mathrm{osc}_{0,K}
\eq
p^{3/2}\sqrt{\mmMK}\omega \frac{h_K}{p}\|\BJ-\BJ_h\|_K,
\qquad
\mathrm{osc}_{\ddiv,K}
\eq
p^{3/2}\frac{1}{\sqrt{\eemK}}\frac{h_K}{p}\|\div(\BJ-\BJ_h)\|_K
\end{align*}
for all $K \in \CT_h$, as well as their ``patchwise'' counterparts
\begin{align*}
\mathrm{osc}_{0,\CT}^2 \eq \sum_{K\in\CT}\mathrm{osc}_{0,K}^2,
\quad
\mathrm{osc}_{\ddiv,\CT}^2 \eq \sum_{K\in\CT}\mathrm{osc}_{\ddiv,K}^2,
\quad
\mathrm{osc}_{\CT}^2 \eq \mathrm{osc}_{0,\CT}^2+\mathrm{osc}_{\ddiv,\CT}^2
\end{align*}
for $\CT \subset \CT_h$. If $\BJ$ is piecewise smooth, this oscillation term decreases
faster than the discretization error to zero. Specifically, for $\CT \subset \CT_h$, if
$\BJ \in \BH^{q+1}(\CT)$ and $\div \BJ \in H^{q+1}(\CT)$ for some
$q \leq p$, then
\begin{equation*}
\osc_{\CT} \lesssim p^{3/2} \left (\frac{h}{p}\right )^{q+1},
\end{equation*}
see \cite{Ern_Gud_Sme_Voh_loc_glob_div_21}.

\subsection{Regular decomposition}
\label{sec_gradient_extraction}

Following \cite[Theorem 2.1]{hiptmair_pechstein_2019}, for all
$\bthe \in \BH_0(\ccurl,\Omega)$, there exist $\bphi \in \BH^1_0(\Omega)$
and $r \in H^1_0(\Omega)$ such that $\bthe = \bphi + \grad r$ with
\begin{equation}
\label{eq_regularity_estimate}
\|\grad \bphi\|_{\cc,\Omega}
\leq
\cR
\left (
\ell_\Omega^{-1} \|\bthe\|_{\cc,\Omega}
+
\|\curl \bthe\|_{\cc,\Omega}
\right ),
\qquad
\|\grad r\|_{\ee,\Omega} \leq \cR \|\bthe\|_{\ee,\Omega},
\end{equation}
where $\cR$ is a constant possibly depending on the geometry of $\Omega$
and, because the results in \cite{hiptmair_pechstein_2019} are
established in non-weighted norms, the material contrasts $\eeMO/\eemO$ and $\mmMO/\mmmO$.

\subsection{Weber's inequality in simply connected domains}
In the particular case where $\Omega$ is simply connected, we have
\begin{equation}
\label{eq_weber_inequality}
\|\bphi\|_{\cc,\Omega} \leq \cW \ell_\Omega \|\curl \bphi\|_{\cc,\Omega}
\qquad
\forall \bphi \in \BH_0(\ccurl,\Omega) \cap \BH(\ddiv^0,\Omega)
\end{equation}
where $\cW$ only depends on the geometry of $\Omega$ and $\mmMO/\mmmO$.

\subsection{Minimal frequency in non-simply connected domains}
If $\Omega$ is not simply connected, we will assume that $\omega \geq \omega_0 > 0$.
Then, we have
\begin{equation}
\label{eq_freq_inequality}
\ell_\Omega^{-1} \|\bphi\|_{\cc,\Omega}
\leq
\left (\frac{\omega_0\ell_\Omega}{c_{\Omega,\max}}\right )^{-1} \omega \|\bphi\|_{\ee,\Omega}
\qquad
\forall \bphi \in \BL^2(\Omega).
\end{equation}

\subsection{Well-posedness}

In the remainder of this document, we will
assume that the (adjoint) problem we consider is well-posed.

\begin{assumption}[Well-posedness]
\label{assumption_well_posedness}
For all $\bj \in \BL^2(\Omega)$, there exists a unique
$\be^\star(\bj) \in \BH_0(\ccurl,\Omega)$ such that
\begin{equation}
\label{eq_def_be_bj}
b(\bw,\be^\star(\bj)) = \omega(\bw,\ee\bj) \qquad \forall \bw \in \BH_0(\ccurl,\Omega).
\end{equation}
\end{assumption}

We can then introduce the notation
\begin{equation}
\label{eq_cS}
\cSE
\eq
\sup_{\substack{\bj \in \BH(\ddiv^0,\ee,\Omega) \\ \|\bj\|_{\ee,\Omega} = 1}}
\enorm{\be^\star(\bj)}.
\end{equation}

\begin{remark}[Well-posedness]
When $\ee$ and $\cc$ are real-valued, it is well-known that \eqref{eq_def_be_bj}
is well-posed if and only if $\omega$ is not a resonance frequency. Hence, the
problem is well-posed except for a countable set of exceptional values in this case.
If $\ee$ and $\cc$ incorporate a sufficiently accurate PML, we may then expect that
\eqref{eq_def_be_bj} is well-posed for all $\omega > 0$ since the scattering problem
with a radiation condition always admits a unique solution \cite{ball_capdebosq_tseringxiao_2011a}.
Besides, \eqref{eq_def_be_bj} is well-posed for all $\omega > 0$ in geophysical
applications where $\ee$ has a definite imaginary part and $\cc$ is real
\cite{chaumontfrelet_nicaise_pardo_2018a}.
\end{remark}

\subsection{Approximation factor}
\label{section_approximation_factor}

We shall also need an ``approximation factor'', that describes
the ability of the discrete space $\BW_h$ to approximate
solutions to \eqref{eq_def_be_bj}, and is defined by
\begin{equation}
\label{eq_gba_second_order}
\gbaE \eq
\sup_{\substack{\bj \in \BH(\ddiv^0,\ee,\Omega) \\ \|\bj\|_{\ee,\Omega} = 1}}
\inf_{\be_h \in \BW_h} \enorm{\be^\star(\bj)-\be_h}.
\end{equation}
Recalling \eqref{eq_cS}, we clearly have $\gbaE \leq \cSE$, 
showing that the approximation factor is bounded independently of the mesh size $h$
and the approximation order $p$. It does however, in general, depend on the
wavenumber $\omega \ell_\Omega/c_{\Omega,\min}$, the geometry of $\Omega$, and the coefficients
$\ee$ and $\mm$.

The divergence-free condition of the right-hand side
defining $\gbaE$ shows that
$\be^\star(\bj) \in \BH_0(\ccurl,\Omega) \cap \BH(\ddiv^0,\overline{\ee},\Omega)$.
Then, as long as $\Omega$ and the coefficients $\ee$ and $\mm$ exhibit a regularity shift
\cite{bonito_guermond_luddens_2013a,costabel_dauge_nicaise_1999a}, we have
$\be^\star(\bj)
\in \BH^s(\Omega)$ for some $0 < s \leq 1$, and standard approximation properties of N\'ed\'elec
spaces (see e.g. \cite{ern_guermond_2017a}) imply that
\begin{equation}
\label{eq_gbaE_gbaH_implicit}
\gbaE
\leq C(\omega,\Omega,\ee,\mm,\beta,s) \left (\frac{h}{p}\right )^s,
\end{equation}
where the constant is independent of $h$ and $p$.
As a result, the approximation factor may depend on the frequency (and more generally, on the geometry of the domain and
the physical coefficients) for coarse meshes, but is asymptotically bounded
independently of the frequency, the mesh, and the approximation order.

In \eqref{eq_gbaE_gbaH_implicit}, the dependence on $\omega$ is not specified.
In addition, the interest of high-order schemes is not obvious, as the convergence
rate is limited to $1$. For Helmholtz problems, several works giving a precise
description of the behavior of the approximation factor are available
\cite{chaumontfrelet_nicaise_2019a,melenk_sauter_2011a}. We also refer the
reader to \cite{chaumontfrelet_vega_2021a,melenk_sauter_2020a,nicaise_tomezyk_2020}
for recent works on Maxwell's equations.

\subsection{Notation for generic constants}

In the remaining of this document, if $A,B \geq 0$ are two positive real values,
we employ the notation $A \lesssim B$ if there exists a constant $C$ that only
depends on $\cI$, $\cb$, $\cR$, $\eps_{\Omega,\max}/\eps_{\Omega,\min}$,
$\mu_{\Omega,\max}/\mu_{\Omega,\min}$, $\cW$ and,
if $\Omega$ is not simply connected, on $\omega_0\ell_\Omega/c_{\Omega,\max}$,
such that $A \leq CB$. Importantly, $C$ is independent of $\omega$, $h$, and $p$.
However, $C$ may depend on $\Omega$ through $\cR$ and $\cW$, and it may
also depend on $\beta$ through $\cI$ and $\cb$. We also employ the notation $A \gtrsim B$
if $B \lesssim A$ and $A \sim B$ if $A \lesssim B$ and $A \gtrsim B$.

\section{A posteriori error estimates}
\label{sec_IPDG}

This section presents our theoretical developments.

\subsection{Numerical scheme}

We consider a discretization space
$\BW_h \subset \BX_h \subset \BCP_p(\Omega) \cap \BH_0(\ccurl,\Omega)$.
In the case where $\BX_h = \BW_h$, the analyzed method corresponds to
the usual first family of N\'ed\'elec finite elements,
while if $\BX_h = \BCP_p(\Omega) \cap \BH_0(\ccurl,\Omega)$,
the second family of N\'ed\'elec elements is covered. As it does
not bring any additional complexity, our analysis
also handles every situation ``in between'', where
the first family of elements is used in some part of the mesh, while the
second family is employed in the remaining areas.

The discrete method reads: find $\BE_h \in \BX_h$ such that
\begin{equation}
\label{eq_maxwell_discrete}
b(\BE_h,\bv_h) = i\omega(\BJ,\bv_h)
\end{equation}
for all $\bv_h \in \BX_h$.
The existence and uniqueness of $\BE_h$ are ensured provided that
Assumption \ref{assumption_well_posedness} holds true and that
the mesh is sufficiently refined
\cite{chaumontfrelet_2019a,%
chaumontfrelet_nicaise_pardo_2018a,%
ern_guermond_2018a,%
zhong_shu_wittum_xu_2009a}.
For general meshes, to the best of our knowledge,
the well-posedness of \eqref{eq_maxwell_discrete} is an open question.

In the remaining of this section, we will work under the assumption that
a discrete solution $\BE_h$ satisfying \eqref{eq_maxwell_discrete} has been
computed. Importantly, our analysis applies as soon as $\BE_h \in \BX_h$
satisfies \eqref{eq_maxwell_discrete} irrespectively of the mesh size
or polynomial degree. In particular, unique solvability
is not required. As a consequence of this assumption, we have
\begin{equation}
\label{eq_galerkin_orthogonality}
b(\BE-\BE_h,\bv_h) = 0
\end{equation}
for all $\bv_h \in \BX_h$.

We point out that the above assumption is not an important limitation in
practice. Indeed, for the simpler setting of the scalar Helmholtz equation, it was
recently established under similar assumptions that $h$-adaptive schemes
converge optimally starting from arbitrarily coarse meshes \cite{bespalov_haberl_praetorius_2017a}.
The idea (see \cite[Algorithm 7]{bespalov_haberl_praetorius_2017a}) is that if
\eqref{eq_maxwell_discrete} is not well-posed, a global uniform refinement is
performed instead of using the error estimator.

\subsection{Error estimators}

The design of our estimator is ``classical'' since we
largely follow \cite{beck_hiptmair_hoppe_wohlmuth_2000a,nicaise_creuse_2003a}.
However, we take special care to weigh each term of the estimator correctly
in terms of the frequency and the electromagnetic coefficients in the spirit of
\cite{cochezdhondt_nicaise_2007}.
Our estimator splits into two parts, namely, for $K \in \CT_h$,
$\eta_K^2 \eq \eta_{\ddiv,K}^2 + \eta_{\ccurl,K}^2$, with
\begin{equation}
\label{eq_definition_eta_div_IPDG}
\eta_{\ddiv,K}
\eq
\frac{1}{\sqrt{\eemK}} \left (
\frac{h_K}{p} \|\div(\BJ-i\omega\ee\BE_h)\|_K
+
\omega \sqrt{\frac{h_K}{p}}\|\jmp{\ee\BE_h} \cdot \bn\|_{\partial K \setminus \partial \Omega}
\right )
\end{equation}
and
\begin{multline}
\label{eq_definition_eta_curl_IPDG}
\eta_{\ccurl,K}
\eq
\sqrt{\mmMK}
\Bigg (
\frac{h_K}{p} \|i\omega\BJ + \omega^2 \ee \BE_h-\curl(\cc\curl \BE_h)\|_K\\
+
\sqrt{\frac{h_K}{p}}\|\jmp{\cc \curl \BE_h} \times \bn\|_{\partial K \setminus \partial \Omega}\Bigg ).
\end{multline}
We also set
\begin{equation*}
\eta^2 \eq \sum_{K \in \CT_h} \eta_K^2,
\qquad
\eta_{\ddiv}^2 \eq \sum_{K \in \CT_h} \eta_{\ddiv,K}^2,
\qquad
\eta_{\ccurl}^2 \eq \sum_{K \in \CT_h} \eta_{\ccurl,K}^2.
\end{equation*}

\subsection{Reliability}

This section is devoted to reliability estimates. The first step consists in controlling
``residual terms'', that is, the sesquilinear form applied to the error and an arbitrary
test function. This is carried out in Lemma \ref{lemma_residual_IPDG_helmholtz} and
\ref{lemma_residual_IPDG} below.

\begin{lemma}[Control of the residual]
\label{lemma_residual_IPDG_helmholtz}
The estimates
\begin{equation}
\label{eq_upper_bound_b_grad_IPDG}
|b(\BE-\BE_h,\grad q)| \lesssim \omega \eta_{\ddiv} \|\grad q\|_{\ee,\Omega}
\end{equation}
and
\begin{equation}
\label{eq_upper_bound_b_H1_IPDG}
|b(\BE-\BE_h,\bphi)| \lesssim \eta_{\ccurl} \|\grad \bphi\|_{\cc,\Omega}
\end{equation}
hold true for all $q \in H^1_0(\Omega)$ and $\bphi \in \BH^1_0(\Omega)$.
\end{lemma}

\begin{proof}
We first establish \eqref{eq_upper_bound_b_grad_IPDG}.
We observe that for any $w \in H^1_0(\Omega)$, we have
\begin{equation*}
b(\BE-\BE_h,\grad w)
=
(i\omega\BJ+\omega^2\ee\BE_h,\grad w)
=
i\omega(\BJ-i\omega\ee\BE_h,\grad w)
\end{equation*}
so that
\begin{align*}
\frac{1}{i\omega} b(\BE-\BE_h,\grad w)
&=
-i\omega\langle \ee \BE_h\cdot \bn,w \rangle_{\partial \CT_h}
-(\div(\BJ-i\omega\ee \BE_h),w)_{\CT_h}
\\
&=
-i\omega \langle \jmp{\ee\BE_h} \cdot \bn,w \rangle_{\CF_h^{\rm i}}
-(\div(\BJ-i\omega\ee\BE_h),w)_{\CT_h},
\end{align*}
and therefore
\begin{align}
\label{tmp_upper_bound_b_grad_IPDG}
\frac{1}{\omega}|b(\BE-\BE_h,\grad w)|
&\leq
\omega\sum_{F \in \CF_h^{\rm i}} \|\jmp{\ee\BE_h}\cdot \bn\|_F\|w\|_F
+
\sum_{K \in \CT_h}\|\div(\BJ-i\omega\ee\BE_h)\|_K\|w\|_K
\\
\nonumber
&\lesssim
\sum_{K \in \CT_h}
\sqrt{\varepsilon_{\widetilde K,\min}}
\eta_{\ddiv,K}
\left (
\frac{p}{h_K} \|w\|_K
+
\sqrt{\frac{p}{h_K}} \|w\|_{\partial K \setminus \partial \Omega}
\right ).
\end{align}
Let now $q \in H^1_0(\Omega)$. Since $\grad (\CQ_h q) \in \BW_h$, by Galerkin orthogonality
\eqref{eq_galerkin_orthogonality}, we can apply \eqref{tmp_upper_bound_b_grad_IPDG} with
$w = q - \CQ_h q$ to show that
\begin{align*}
\frac{1}{\omega}|b(\BE-\BE_h,\grad q)|
&=
\frac{1}{\omega}|b(\BE-\BE_h,\grad (q-\CQ_h q))|
\\
&\lesssim
\sum_{K \in \CT_h}
\sqrt{\varepsilon_{\widetilde K,\min}}
\eta_{\ddiv,K}
\left (
\frac{p}{h_K} \|q-\CQ_h q\|_K
+
\sqrt{\frac{p}{h_K}} \|q-\CQ_h q\|_{\partial K \setminus \partial \Omega}
\right )
\\
&\lesssim
\sum_{K \in \CT_h} \sqrt{\varepsilon_{\widetilde K,\min}} \eta_{\ddiv,K} \|\grad q\|_{\tK}
\lesssim
\eta_{\ddiv} \|\grad q\|_{\ee,\Omega},
\end{align*}
where we additionally employed \eqref{eq_quasi_interpolation_H1}. This shows
\eqref{eq_upper_bound_b_grad_IPDG}.

We now focus on \eqref{eq_upper_bound_b_H1_IPDG}. We start with an arbitrary element
$\bw \in \BH^1(\CT_h) \cap \BH_0(\ccurl,\Omega)$. We have
\begin{align*}
b(\BE-\BE_h,\bw)
&=
i\omega (\BJ,\bw)+\omega^2(\ee\BE_h,\bw)-(\cc \curl \BE_h,\curl \bw)
\\
&=
(i\omega \BJ + \omega^2 \BE_h -\curl (\cc\curl \BE_h),\bw)_{\CT_h}
+
\langle \cc \curl \BE_h \times \bn,\bw \rangle_{\partial \CT_h}
\\
&=
(i\omega \BJ + \omega^2 \BE_h -\curl (\cc\curl \BE_h),\bw)_{\CT_h}
+
\langle \jmp{\cc \curl \BE_h} \times \bn,\bw \rangle_{\CF_h^{\rm i}}
\end{align*}
and
\begin{align}
\label{tmp_upper_bound_b_H1_IPDG}
|b(\BE-\BE_h,\bw)|
&\leq
\sum_{K \in \CT_h} \|i\omega \BJ + \omega^2 \BE_h -\curl (\cc\curl \BE_h)\|_K\|\bw\|_K
\\
\nonumber
&\quad+
\sum_{F \in \CF_h^{\rm i}}
\|\jmp{\cc \curl \BE_h} \times \bn\|_F\|\bw\times \bn\|_F
\\
\nonumber
&\lesssim
\sum_{K \in \CT_h}
\frac{1}{\sqrt{\mmMK}}
\eta_{\ccurl,K}
\left (\frac{p}{h_K} \|\bw\|_K
+
\sqrt{\frac{p}{h_K}} \|\bw \times \bn\|_{\partial K\setminus\partial\Omega}\right ).
\end{align}
Then, if $\bphi \in \BH^1_0(\Omega)$, since $\CR_h \bphi \in \BW_h$,
we can employ Galerkin orthogonality
\eqref{eq_galerkin_orthogonality},
\eqref{tmp_upper_bound_b_H1_IPDG} with $\bw = \bphi - \CR_h \bphi$
and \eqref{eq_quasi_interpolation_Hc}, showing that
\begin{align*}
|b(\BE-\BE_h,\bphi)| &= |b(\BE-\BE_h,\bphi-\CR_h\bphi)|
\\
&\lesssim\sum_{K \in \CT_h} \frac{1}{\sqrt{\mmMK}} \eta_{\ccurl,K}
\left (
\frac{p}{h_K} \|\bphi-\CR_h\bphi\|_K
+
\sqrt{\frac{p}{h_K}} \|(\bphi-\CR_h\bphi) \times \bn\|_{\partial K\setminus\partial\Omega}
\right )
\\
&\lesssim
\sum_{K \in \CT_h} \frac{1}{\sqrt{\mmMK}} \eta_{\ccurl,K} \|\grad \bphi\|_{\tK}
\lesssim
\sum_{K \in \CT_h} \eta_{\ccurl,K} \|\grad \bphi\|_{\cc,\tK}
\lesssim
\eta_{\ccurl} \|\grad \bphi\|_{\cc,\Omega}.
\end{align*}
\end{proof}

\begin{lemma}[General control of the residual]
\label{lemma_residual_IPDG}
We have
\begin{equation}
\label{eq_upper_bound_b_div_free_IPDG}
|b(\BE-\BE_h,\bthe)| \lesssim \eta
\enorm{\bthe}
\end{equation}
for all $\bthe \in \BH_0(\ccurl,\Omega)$.
\end{lemma}

\begin{proof}
Let $\bthe \in\BH_0(\ccurl,\Omega)$. We start by introducing
$q$ as the unique element of $H_0^1(\Omega)$ such that
\begin{equation*}
(\grad q,\grad v)_\Omega = (\bthe,\grad v)_\Omega
\qquad
\forall v\in H_0^1(\Omega),
\end{equation*}
so that $\widetilde \bthe \eq \bthe-\grad q \in \BH_0(\ccurl,\Omega) \cap \BH(\ddiv^0,\Omega)$,
and
\begin{equation*}
\omega \|\grad q\|_{\ee,\Omega} \lesssim \omega \|\bthe\|_{\ee,\Omega},
\qquad
\enorm{\widetilde \bthe} \lesssim \enorm{\bthe}.
\end{equation*}
We then invoke Section \ref{sec_gradient_extraction}, which yields the
existence of $r \in H^1_0(\Omega)$ and $\bphi \in \BH^1_0(\Omega)$
such that $\widetilde \bthe = \grad r + \bphi$, with
\begin{equation*}
\|\grad r\|_{\ee,\Omega}
\lesssim
\|\widetilde \bthe\|_{\ee,\Omega}
\lesssim
\|\bthe\|_{\ee,\Omega},
\qquad
\|\grad \bphi\|_{\cc,\Omega}
\lesssim
\ell_\Omega^{-1} \|\widetilde \bthe\|_{\cc,\Omega} + \|\curl \widetilde \bthe\|_{\cc,\Omega}.
\end{equation*}
Employing
\eqref{eq_upper_bound_b_grad_IPDG} and \eqref{eq_upper_bound_b_H1_IPDG}, we have
\begin{align*}
|b(\BE-\BE_h,\bthe)|
&\lesssim
\omega \eta_{\ddiv} \|\grad q\|_{\ee,\Omega}
+
\omega \eta_{\ddiv} \|\grad r\|_{\ee,\Omega}
+
\eta_{\ccurl} \|\grad \bphi\|_{\cc,\Omega}
\\
&\lesssim
\eta
\left (
\omega \|\bthe\|_{\ee,\Omega}
+
\ell_\Omega^{-1} \|\widetilde \bthe\|_{\cc,\Omega}
+
\|\curl \widetilde \bthe\|_{\cc,\Omega}
\right ).
\end{align*}
If $\Omega$ is simply-connected, we employ Weber's inequality
\eqref{eq_weber_inequality}, and we write that
\begin{equation*}
\ell_\Omega^{-1} \|\widetilde \bthe\|_{\cc,\Omega} + \|\curl \widetilde \bthe\|_{\cc,\Omega}
\lesssim
(1+\cW) \|\curl \widetilde \bthe\|_{\cc,\Omega}
\lesssim
\enorm{\widetilde \bthe}
\lesssim
\enorm{\bthe}.
\end{equation*}
When $\Omega$ is not simply-connected, we utilize that
\begin{equation*}
\ell_\Omega^{-1} \|\widetilde \bthe\|_{\cc,\Omega} + \|\curl \widetilde \bthe\|_{\cc,\Omega}
\lesssim
\left (\frac{\omega_0\ell_\Omega}{c_{\Omega,\max}}\right )^{-1}
\omega \|\widetilde \bthe\|_{\ee,\Omega} + \|\curl \widetilde \bthe\|_{\cc,\Omega}
\lesssim
\enorm{\widetilde \bthe}
\lesssim
\enorm{\bthe}.
\end{equation*}
In both cases, this concludes the proof.
\end{proof}

The next step of the proof is an ``Aubin-Nitsche'' type result that we employ to estimate
the error measured in the $\BL^2(\Omega)$-norm. This step is required to make up for the
lack of coercivity of the sesquilinear form $b$.

\begin{lemma}[Aubin-Nitsche]
\label{AubinNitsche_IPDG}
We have
\begin{equation}
\label{eq_L2_estimate_second_order}
\omega\|\BE-\BE_h\|_{\ee,\Omega}
\lesssim
(1+\gbaE) \eta.
\end{equation}
\end{lemma}

\begin{proof}
We first introduce the Helmholtz decomposition of the error.
Namely, we define $s$ as the unique element of $H^1_0(\Omega)$ such that
\begin{equation*}
(\ee \grad s,\grad v) = (\ee(\BE-\BE_h),\grad v) \qquad \forall v \in H^1_0(\Omega),
\end{equation*}
so that $\BE-\BE_h = \grad s + \bthe$, with $\bthe \in \BH(\ddiv^0,\ee,\Omega)$.
For the gradient part, thanks to \eqref{eq_upper_bound_b_grad_IPDG}, we have
\begin{equation*}
\omega^2\|\grad s\|_{\ee,\Omega}^2
=
\omega^2\Re(\ee(\BE-\BE_h),\grad s)
=
-\Re b(\BE-\BE_h,\grad s)
\lesssim
\omega\eta_{\ddiv}\|\grad s\|_{\ee,\Omega},
\end{equation*}
so that
$\omega \|\grad s\|_{\ee,\Omega} \lesssim \eta_{\ddiv}$.

For the divergence-free part, let $\bxi$ be the unique element of
$\BH_0(\ccurl,\Omega)$ such that $b(\bw,\bxi) = \omega(\bw,\ee \bthe)$
for all $\bw \in \BH_0(\ccurl,\Omega)$. Using \eqref{eq_upper_bound_b_div_free_IPDG},
we have
\begin{align*}
\omega \Re(\bthe,\ee \bthe)
=
\omega\Re(\BE-\BE_h,\ee\bthe)
=
\Re b(\BE-\BE_h,\bxi)´
=
\Re b(\BE-\BE_h,\bxi-\bxi_h)
\lesssim
\eta\enorm{\bxi-\bxi_h}
\end{align*}
for all $\bxi_h \in \BW_h$. As $\bthe \in \BH(\ddiv^0,\ee,\Omega)$,
recalling definition \eqref{eq_gba_second_order} of the approximation factor,
it holds that
\begin{equation*}
\omega \|\bthe\|_{\ee,\Omega}^2
=
\omega \Re (\bthe,\ee\bthe)
\lesssim
\gbaE\eta\|\bthe\|_{\ee,\Omega},
\end{equation*}
and \eqref{eq_L2_estimate_second_order} follows.
\end{proof}

We close this section with the reliability estimate, that uses G\aa rding inequality
\eqref{eq_garding_inequality}.

\begin{theorem}[Reliability]
\label{rel_IPDG}
The following estimate holds true
\begin{equation*}
\enorm{\BE-\BE_h} \lesssim (1 + \gbaE) \eta.
\end{equation*}
\end{theorem}

\begin{proof}
Using \eqref{eq_upper_bound_b_div_free_IPDG}, we have
\begin{equation}
\label{rel_aux_IPDG}
\Re b(\BE-\BE_h,\BE-\BE_h) \lesssim \eta \enorm{\BE-\BE_h}.
\end{equation}
On the other hand, by using G\aa rding inequality \eqref{eq_garding_inequality}, we have
\begin{align}
\label{rel_Garding_IPDG}
\enorm{\BE-\BE_h}^2
\lesssim
\Re b(\BE-\BE_h,\BE-\BE_h) + \omega^2\|\BE-\BE_h\|_{\ee,\Omega}^2,
\end{align}
and the result follows from \eqref{rel_aux_IPDG}, \eqref{rel_Garding_IPDG},
Lemma \ref{AubinNitsche_IPDG} and Young's inequality.
\end{proof}

\subsection{Efficiency} 

In this section, we focus on efficiency estimates. Classically, the proofs rely on
``bubble functions'' and their key properties introduced in Section \ref{bubbles}.

\begin{lemma}\label{eff_eta_div_IPDG}
We have
\begin{equation}
\label{eq_efficiency_eta_div_IPDG}
\eta_{\ddiv,K}
\lesssim
p^{3/2}
\omega\|\BE-\BE_h\|_{\ee,\tK}
+
\mathrm{osc}_{\ddiv,\CT_{K,h}}
\qquad \forall K \in \CT_h.
\end{equation}
\end{lemma}

\begin{proof}
Let $K \in \CT_h$, $r_K \eq \div (\BJ_h-i\omega\ee\BE_h)$ and $v_K \eq b_K r_K$.
Integrating by parts, we have
\begin{align*}
\|b_K^{1/2}r_K\|_K^2= (r_K,v_K)_K
&=
-i\omega(\ee(\BE-\BE_h),\grad v_K)_K - (\div(\BJ-\BJ_h),v_K)_K
\\
&\leq
\omega\|\ee(\BE-\BE_h)\|_K\|\grad v_K\|_K
+
\|b_K^{1/2}\div(\BJ-\BJ_h)\|_K\|b_K^{1/2}r_K\|_K.
\end{align*}
Recalling \eqref{eq_norm_bubble} and \eqref{eq_inv_bubble}, we have
$\|r_K\|_K\lesssim p\|b_K^{1/2}r_K\|_K $ and
$\|\grad v_K\|_K \lesssim (p/h_K) \|b_K^{1/2}r_K\|_K$, and it follows that
\begin{equation}
\label{tmp_eff_div_K}
\frac{1}{\sqrt{\varepsilon_{\tK,\min}}}\frac{h_K}{p}
\|\div(\BJ-i\omega\BE_h)\|_{K}
\lesssim
p\omega \|\BE-\BE_h\|_{\ee,K} +
\frac{p}{\sqrt{\varepsilon_{\tK,\min}}}\frac{h_K}{p} \|\div(\BJ-\BJ_h)\|_K.
\end{equation}
On the other hand, for $F \in \CF_h^{\rm i} \cap \CF_K$, we set
$r_F \eq \omega \jmp{\ee\BE_h} \cdot \bn_F$ and $v_F \eq \LE(r_F)$.
We have
\begin{align*}
\|b_F^{1/2}r_F\|_F^2= (r_F,v_F)_F
&=
|(\div(i\omega \ee\BE_h),v_F)_{\CT_{F,h}} + i\omega (\ee\BE_h,\grad v_F)_{\CT_{F,h}}|
\\
&=
|(\div(\BJ-i\omega \ee\BE_h),v_F)_{\CT_{F,h}} + i\omega (\ee(\BE-\BE_h),\grad v_F)_{\CT_{F,h}}|
\\
&\leq
\|\div(\BJ-i\omega\ee\BE_h)\|_{\CT_{F,h}}\|v_F\|_{\widetilde F}
+
\omega\|\ee(\BE-\BE_h)\|_{\widetilde F}\|\grad v_F\|_{\widetilde F}.
\end{align*}
Since $\|v_F\|_{\widetilde F} \lesssim (\sqrt{h_F}/p) \|b_F^{1/2}r_F\|_F$
and $\|\grad v_F\|_{\widetilde F} \lesssim (p/\sqrt{h_F}) \|b_F^{1/2}r_F\|_F$
from \eqref{eq_ext_bubble}, we conclude that
\begin{multline}
\label{tmp_eff_div_F}
\frac{1}{\sqrt{\varepsilon_{\tK,\min}}}
\sqrt{\frac{h_F}{p}}\|\jmp{\ee\BE_h} \cdot \bn_F\|_F
\lesssim
\frac{1}{\sqrt{\varepsilon_{\tK,\min}}} \sqrt{\frac{h_F}{p}}p\|b_F^{1/2}r_F\|_F
\\
\lesssim
\frac{p^{1/2}}{\sqrt{\varepsilon_{\tK,\min}}}
\frac{h_F}{p} \|\div(\BJ-i\omega\ee\BE_h)\|_{\CT_{F,h}}
+
p^{1/2}\omega \|\BE-\BE_h\|_{\ee,\tF}.
\end{multline}
Then, \eqref{eq_efficiency_eta_div_IPDG} follows from
\eqref{tmp_eff_div_K} and \eqref{tmp_eff_div_F}.%
\end{proof}

\begin{lemma}
\label{eff_eta_curl_IPDG}
We have
\begin{equation}
\label{eq_efficiency_eta_curl_IPDG}
\eta_{\ccurl,K}
\lesssim
p^{3/2}\left(1+\frac{\omega h_K}{pc_{\tK,\min}}\right) 
\norm{\BE-\BE_h}_{\ccurl,\omega,\tK} + \osc_{0,\CT_{K,h}}
\quad
\forall K \in \CT_h.
\end{equation}
\end{lemma}

\begin{proof}
Let $K \in \CT_h$, $\br_K \eq i\omega\BJ_h-\curl(\cc\curl\BE_h)+\omega^2\ee\BE_h$
and $\bv_K \eq b_K \br_K$. After integrating by parts, we have
\begin{align*}
\|b_K^{1/2}\br_K\|_K^2
=
(\br_K,\bv_K)_{K}
&=
(i\omega\BJ_h-\curl(\cc\curl\BE_h)+\omega^2\ee\BE_h,\bv_K)_K
\\
&=
(\cc\curl(\BE-\BE_h),\curl\bv_K)_K-(\omega^2\ee(\BE-\BE_h),\bv_K)_K-(i\omega(\BJ-\BJ_h),\bv_K)_K
\\
&\leq
\|\cc\curl(\BE-\BE_h)\|_K\|\curl\bv_K\|_K
+
\omega^2\|b_K^{1/2}\ee(\BE-\BE_h)\|_K\|b_K^{1/2}\br_K\|_K
\\
&\quad+\omega\|b_K^{1/2}(\BJ-\BJ_h)\|_K\|b_K^{1/2}\br_K\|_K.
\end{align*}
Thanks to \eqref{eq_norm_bubble} and \eqref{eq_inv_bubble}, we see that
$\|\curl\bv_K\|_K \leq \|\grad\bv_K\|_K \lesssim (p/h_K)\|b_K^{1/2}\br_K\|_K$
and $\|\br_K\|_K\lesssim p\|b_K^{1/2}\br_K\|_K$. Then, we find
\begin{align*}
\sqrt{\mu_{\tK,\max}}\frac{h_K}{p}\|\br_K\|_K
\lesssim
p \left (\|\curl(\BE-\BE_h)\|_{\cc,K}
\!+\!
\frac{\omega h_K}{p c_{\tK,\min}}\omega\|\BE-\BE_h\|_{\ee,K}
\!+\!
\sqrt{\mu_{\tK,\max}}\omega\frac{h_K}{p}\|\BJ-\BJ_h\|_K \right ),
\end{align*}
and hence
\begin{multline}
\label{eta_curl_K}
\sqrt{\mu_{\tK,\max}}\frac{h_K}{p}\|i\omega\BJ-\curl(\cc\curl\BE_h)+\omega^2\ee\BE_h\|_K
\\
\lesssim
p \left(1 + \frac{\omega h_K}{pc_{\tK,\min}}\right )\norm{\BE-\BE_h}_{\ccurl,\omega,K}
+
p^{-1/2}\osc_K.
\end{multline}
On the other hand, for $F \in \CF_h^{\rm i} \cap \CF_K$,
we set $\br_F \eq \jmp{\cc\curl\BE_h} \times \bn_F$ and $\bv_F \eq \LE(\br_F)$. Then,
\begin{align*}
&\|b_F^{1/2}\br_F\|_F^2
=
(\br_F,\bv_F)_F
\\
&=
|(\curl (\cc\curl \BE_h),\bv_F)_{\CT_{F,h}} - (\cc \curl \BE_h,\curl \bv_F)_{\CT_{F,h}}|
\\
&=
|(\curl (\cc\curl (\BE-\BE_h)),\bv_F)_{\CT_{F,h}}
-
(\cc \curl (\BE-\BE_h),\curl \bv_F)_{\CT_{F,h}}|
\\
&=
|(i\omega\BJ+\omega^2 \ee\BE -\curl (\cc\curl \BE_h),\bv_F)_{\CT_{F,h}}
-
(\cc \curl (\BE-\BE_h),\curl \bv_F)_{\CT_{F,h}}|
\\
&=
|(i\omega\BJ+\omega^2 \ee\BE_h -\curl (\cc\curl \BE_h),\bv_F)_{\CT_{F,h}}
+
\omega^2(\ee(\BE-\BE_h),\bv_F)_{\CT_{F,h}}
-
(\cc \curl (\BE-\BE_h),\curl \bv_F)_{\CT_{F,h}}|.
\end{align*}
Recalling \eqref{eq_ext_bubble},
$\|\bv_F\|_{\tF} \lesssim (\sqrt{h_F}/p)\|b_F^{1/2}\br_F\|_F$
and
$\|\curl \bv_F\|_{\tF} \lesssim (p/\sqrt{h_F}) \|b_F^{1/2}\br_F\|_F$,
which leads to
\begin{multline*}
\sqrt{\mu_{\tK,\max}} \sqrt{\frac{h_K}{p}} \|\br_F\|_F
\lesssim \sqrt{\mu_{\tK,\max}} \sqrt{\frac{h_K}{p}} p\|b_F^{1/2}\br_F\|_F
\lesssim
\\
p\left (
\sqrt{\frac{\mu_{\tK,\max}}{p}}\frac{h_F}{p} \|i\omega \BJ+\omega^2\ee\BE_h-\curl(\cc\curl \BE_h)\|_{\CT_{F,h}}
+
\frac{1}{\sqrt{p}}\frac{\omega h_F}{pc_{\tF,\min}} \omega \|\BE-\BE_h\|_{\ee,\tF}
+
\sqrt{p} \|\curl(\BE-\BE_h)\|_{\cc,\tF}
\right ),
\end{multline*}
and it follows that
\begin{multline}
\label{eta_curl_F}
\sqrt{\mu_{\tK,\max}}\sqrt{\frac{h_K}{p}}
\|\jmp{\cc\curl\BE_h}\times \bn_F\|_F
\lesssim
\\
p^{3/2}\left (
1 + \frac{\omega h_F}{p c_{\tF,\min}}
\right )
\norm{\BE-\BE_h}_{\ccurl,\omega,\tF}
+
p^{1/2}\sqrt{\mu_{\tK,\max}}\frac{h_F}{p} \|i\omega \BJ+\omega^2\ee\BE_h-\curl(\cc\curl \BE_h)\|_{\CT_{F,h}}.
\end{multline}
Thus, \eqref{eq_efficiency_eta_curl_IPDG} follows from \eqref{eta_curl_K} and \eqref{eta_curl_F}.%
\end{proof}
	
We now state our efficiency estimate, that is a direct consequence of Lemmas
\ref{eff_eta_div_IPDG} and \ref{eff_eta_curl_IPDG}.
	
\begin{theorem}[Efficiency]\label{eff_IPDG}
We have
\begin{equation*}
\eta_K
\lesssim
p^{3/2}\left(1+\frac{\omega h_K}{pc_{\tK,\min}}\right) 
\tnorm{\BE-\BE_h}_{\ccurl,\omega,\tK} + \osc_{\CT_{K,h}}
\qquad
\forall K \in \CT_h.
\end{equation*}
\end{theorem}

\section{Numerical experiments}
\label{sec_numerics}

In this section, we present three numerical examples in 2D with N\'ed\'elec elements
of the first family that illustrate our main findings. While the previous
analysis was rigorously carried out in 3D, the key results also apply in 2D
with the usual modifications for the curl operator.

In Experiments \ref{experiment_cavity} and \ref{experiment_pml}, we employ
structured meshes. Given $h \eq 1/n$, these meshes are defined by first
introducing a $n \times n$ Cartesian grid of the domain and then splitting
each square of the grid into four triangles by joining the barycenter of
the square with each of its vertices. On the other hand, Experiment
\ref{experiment_scattering} relies on unstructured meshes that are generated
using the {\tt MMG} software package \cite{mmg3d}.

We employ the {\tt LDLT} factorization from {\tt MUMPS} software package
to solve sparse linear systems \cite{amestoy_duff_lexcellent_2000a}.

\subsection{Analytical solution in a PEC cavity}
\label{experiment_cavity}

We consider the square $\Omega \eq (-1,1)^2$, with coefficients $\ee \eq \BI$ and $\mu \eq 1$.
The source term is $\BJ \eq \be_1$, and the corresponding solution to \eqref{eq_maxwell_strong}
reads
\begin{equation}
\label{eq_sol_cavity}
\BE(\bx) = \frac{1}{\omega} \left ( \frac{\cos(\omega \bx_2)}{\cos \omega} - 1\right ) \be_1.
\end{equation}
As the problem under consideration does not feature absorption, there are resonance frequencies
for which it is not well-posed. These resonances are attained at $\omega = k\pi/2$,
$k \in \mathbb N^\star$, which is in agreement with \eqref{eq_sol_cavity}.
This example is especially interesting since the explicit expression
\begin{equation}
\label{eq_estimate_gbaE_simplified}
\gbaE
\leq
C(\Omega,p)
\left (
\frac{\omega h}{c}
+
\frac{\omega}{\delta}
\left (
\frac{\omega h}{c}
\right )^{\!p}
\ \!\right ),
\end{equation}
where $\delta$ is the distance between $\omega$ and closest resonance,
is provided in \cite{chaumontfrelet_vega_2021a} for the approximation factor.

We consider two sequences of frequencies to illustrate the influence of the
approximation factor $\gbaE$. On the one hand, we consider a sequence of frequencies
tending towards the resonance frequency $\omega_{\rm r} = 3\pi/2$. On the other hand,
we also study a sequence of increasing frequencies uniformly separated from the resonance set.

The first sequence of frequencies takes the form $\omega_{\delta} \eq \omega_{\rm r} + \delta(\pi/2)$
with $\delta = 1/2,1/4,1/8,\dots,1/64$. Figure \ref{figure_pec_cavity_resonance} presents the
corresponding results. The second sequence of frequencies reads
$\omega_\ell \eq (\ell+3/10) \times 2\pi$, for $\ell \eq 1,2,4,8,16,32$.
The corresponding results are reproduced in Figure \ref{figure_pec_cavity_high}.

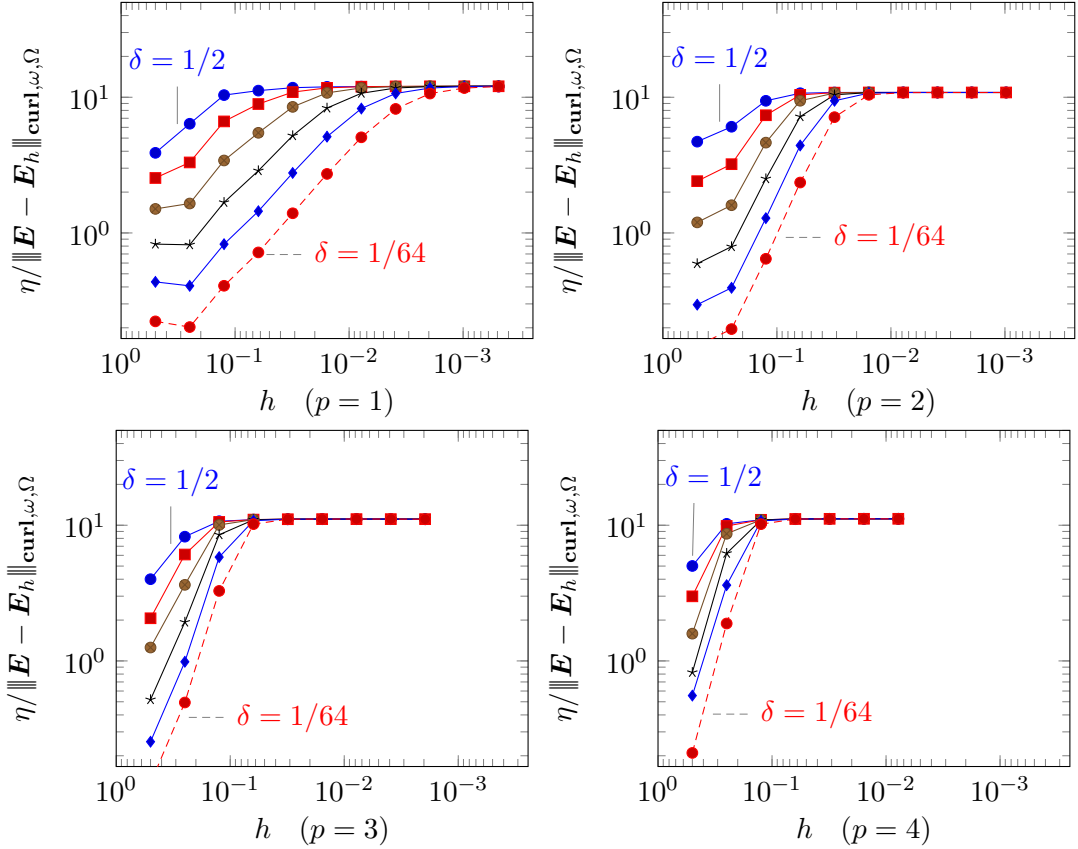
\begin{figure}
\begin{minipage}{.45\linewidth}
\begin{tikzpicture}
\begin{axis}%
[%
	width  = \linewidth,
	xlabel = {$h$ $\;$ ($p = 1$)},
	ylabel = {$\eta/\norm{\BE-\BE_h}_{\ccurl,\omega,\Omega}$},
	xmode = log,
	ymode = log,
	x dir  = reverse,
	ymin = 1./6,
	ymax = 50,
	xmax = 1,
	xmin = 1/4096
]

\plot%
table[x expr = 1./\thisrow{N}, y expr=\thisrow{eta}/\thisrow{err}]%
{figures/data/resonance/P0/curve_1.3.txt}%
node[pos=0.075,pin=90:{$\delta = 1/2$}] {};

\plot%
table[x expr = 1./\thisrow{N}, y expr=\thisrow{eta}/\thisrow{err}]%
{figures/data/resonance/P0/curve_1.275.txt};

\plot%
table[x expr = 1./\thisrow{N}, y expr=\thisrow{eta}/\thisrow{err}]%
{figures/data/resonance/P0/curve_1.2625.txt};

\plot%
table[x expr = 1./\thisrow{N}, y expr=\thisrow{eta}/\thisrow{err}]%
{figures/data/resonance/P0/curve_1.25625.txt};

\plot%
table[x expr = 1./\thisrow{N}, y expr=\thisrow{eta}/\thisrow{err}]%
{figures/data/resonance/P0/curve_1.25313.txt};

\plot%
table[x expr = 1./\thisrow{N}, y expr=\thisrow{eta}/\thisrow{err}]%
{figures/data/resonance/P0/curve_1.25156.txt} %
node[pos=0.3,pin=0:{$\delta = 1/64$}] {};

\end{axis}
\end{tikzpicture}
\end{minipage}
\begin{minipage}{.45\linewidth}
\begin{tikzpicture}
\begin{axis}%
[%
	width  = \linewidth,
	xlabel = {$h$ $\;$ ($p = 2$)},
	ylabel = {$\eta/\norm{\BE-\BE_h}_{\ccurl,\omega,\Omega}$},
	xmode = log,
	ymode = log,
	x dir  = reverse,
	ymin = 1/6,
	ymax = 1/.02,
	xmax = 1,
	xmin = 1/4096
]

\plot%
table[x expr = 1./\thisrow{N}, y expr=\thisrow{eta}/\thisrow{err}]%
{figures/data/resonance/P1/curve_1.3.txt}
node[pos=0.075,pin=90:{$\delta = 1/2$}] {};

\plot%
table[x expr = 1./\thisrow{N}, y expr=\thisrow{eta}/\thisrow{err}]%
{figures/data/resonance/P1/curve_1.275.txt};

\plot%
table[x expr = 1./\thisrow{N}, y expr=\thisrow{eta}/\thisrow{err}]%
{figures/data/resonance/P1/curve_1.2625.txt};

\plot%
table[x expr = 1./\thisrow{N}, y expr=\thisrow{eta}/\thisrow{err}]%
{figures/data/resonance/P1/curve_1.25625.txt};

\plot%
table[x expr = 1./\thisrow{N}, y expr=\thisrow{eta}/\thisrow{err}]%
{figures/data/resonance/P1/curve_1.25313.txt};

\plot%
table[x expr = 1./\thisrow{N}, y expr=\thisrow{eta}/\thisrow{err}]%
{figures/data/resonance/P1/curve_1.25156.txt} %
node[pos=0.3,pin=0:{$\delta = 1/64$}] {};

\end{axis}
\end{tikzpicture}
\end{minipage}

\begin{minipage}{.45\linewidth}
\begin{tikzpicture}
\begin{axis}%
[%
	width  = \linewidth,
	xlabel = {$h$ $\;$ ($p = 3$)},
	ylabel = {$\eta/\norm{\BE-\BE_h}_{\ccurl,\omega,\Omega}$},
	xmode = log,
	ymode = log,
	x dir  = reverse,
	ymin = 1/6,
	ymax = 1/.02,
	xmax = 1,
	xmin = 1/4096
]

\plot%
table[x expr = 1./\thisrow{N}, y expr=\thisrow{eta}/\thisrow{err}]%
{figures/data/resonance/P2/curve_1.3.txt}
node[pos=0.1,pin=90:{$\delta = 1/2$}] {};

\plot%
table[x expr = 1./\thisrow{N}, y expr=\thisrow{eta}/\thisrow{err}]%
{figures/data/resonance/P2/curve_1.275.txt};

\plot%
table[x expr = 1./\thisrow{N}, y expr=\thisrow{eta}/\thisrow{err}]%
{figures/data/resonance/P2/curve_1.2625.txt};

\plot%
table[x expr = 1./\thisrow{N}, y expr=\thisrow{eta}/\thisrow{err}]%
{figures/data/resonance/P2/curve_1.25625.txt};

\plot%
table[x expr = 1./\thisrow{N}, y expr=\thisrow{eta}/\thisrow{err}]%
{figures/data/resonance/P2/curve_1.25313.txt};

\plot%
table[x expr = 1./\thisrow{N}, y expr=\thisrow{eta}/\thisrow{err}]%
{figures/data/resonance/P2/curve_1.25156.txt} %
node[pos=0.15,pin=0:{$\delta = 1/64$}] {};

\end{axis}
\end{tikzpicture}
\end{minipage}
\begin{minipage}{.45\linewidth}
\begin{tikzpicture}
\begin{axis}%
[%
	width  = \linewidth,
	xlabel = {$h$ $\;$ ($p = 4$)},
	ylabel = {$\eta/\norm{\BE-\BE_h}_{\ccurl,\omega,\Omega}$},
	xmode = log,
	ymode = log,
	x dir  = reverse,
	ymin = 1/6,
	ymax = 1/.02,
	xmax = 1,
	xmin = 1/4096
]

\plot%
table[x expr = 1./\thisrow{N}, y expr=\thisrow{eta}/\thisrow{err}]%
{figures/data/resonance/P3/curve_1.3.txt}%
node[pos=0.0,pin={[pin distance=.7cm]88:{\hspace{.5cm}$\delta = 1/2$}}] {};

\plot%
table[x expr = 1./\thisrow{N}, y expr=\thisrow{eta}/\thisrow{err}]%
{figures/data/resonance/P3/curve_1.275.txt};

\plot%
table[x expr = 1./\thisrow{N}, y expr=\thisrow{eta}/\thisrow{err}]%
{figures/data/resonance/P3/curve_1.2625.txt};

\plot%
table[x expr = 1./\thisrow{N}, y expr=\thisrow{eta}/\thisrow{err}]%
{figures/data/resonance/P3/curve_1.25625.txt};

\plot%
table[x expr = 1./\thisrow{N}, y expr=\thisrow{eta}/\thisrow{err}]%
{figures/data/resonance/P3/curve_1.25313.txt};

\plot%
table[x expr = 1./\thisrow{N}, y expr=\thisrow{eta}/\thisrow{err}]%
{figures/data/resonance/P3/curve_1.25156.txt} %
node[pos=0.1,pin=0:{$\delta = 1/64$}] {};

\end{axis}
\end{tikzpicture}
\end{minipage}
\caption{PEC cavity: near resonance example}
\label{figure_pec_cavity_resonance}
\vspace{-0.75cm}
\end{figure}

\begin{figure}
\begin{minipage}{.45\linewidth}
\begin{tikzpicture}
\begin{axis}%
[%
	width  = \linewidth,
	xlabel = {$h$ $\;$ ($p = 1$)},
	ylabel = {$\eta/\norm{\BE-\BE_h}_{\ccurl,\omega,\Omega}$},
	xmode = log,
	ymode = log,
	x dir  = reverse,
	ymax = 50,
	ymin = 0.1666666,
	xmax = 1,
	xmin = 1/4096
]

\plot%
table[x expr = 1./\thisrow{N}, y expr=\thisrow{eta}/\thisrow{err}]%
{figures/data/high/P0/curve_1.3.txt}%
node[pos=0.25,pin=90:{$\ell = 1$}] {};

\plot%
table[x expr = 1./\thisrow{N}, y expr=\thisrow{eta}/\thisrow{err}]%
{figures/data/high/P0/curve_2.3.txt};

\plot%
table[x expr = 1./\thisrow{N}, y expr=\thisrow{eta}/\thisrow{err}]%
{figures/data/high/P0/curve_4.3.txt};

\plot%
table[x expr = 1./\thisrow{N}, y expr=\thisrow{eta}/\thisrow{err}]%
{figures/data/high/P0/curve_8.3.txt};

\plot%
table[x expr = 1./\thisrow{N}, y expr=\thisrow{eta}/\thisrow{err}]%
{figures/data/high/P0/curve_16.3.txt};

\plot%
table[x expr = 1./\thisrow{N}, y expr=\thisrow{eta}/\thisrow{err}]%
{figures/data/high/P0/curve_32.3.txt}%
node[pos=0.9,pin=-90:{$\ell = 32$}] {};

\end{axis}
\end{tikzpicture}
\end{minipage}
\begin{minipage}{.45\linewidth}
\begin{tikzpicture}
\begin{axis}%
[%
	width  = \linewidth,
	xlabel = {$h$ $\;$ ($p = 2$)},
	ylabel = {$\eta/\norm{\BE-\BE_h}_{\ccurl,\omega,\Omega}$},
	xmode = log,
	ymode = log,
	x dir  = reverse,
	ymax = 50,
	ymin = 0.1666666,
	xmax = 1,
	xmin = 1/4096
]

\plot%
table[x expr = 1./\thisrow{N}, y expr=\thisrow{eta}/\thisrow{err}]%
{figures/data/high/P1/curve_1.3.txt}
node[pos=0.15,pin=90:{$\ell = 1$}] {};

\plot%
table[x expr = 1./\thisrow{N}, y expr=\thisrow{eta}/\thisrow{err}]%
{figures/data/high/P1/curve_2.3.txt};

\plot%
table[x expr = 1./\thisrow{N}, y expr=\thisrow{eta}/\thisrow{err}]%
{figures/data/high/P1/curve_4.3.txt};

\plot%
table[x expr = 1./\thisrow{N}, y expr=\thisrow{eta}/\thisrow{err}]%
{figures/data/high/P1/curve_8.3.txt};

\plot%
table[x expr = 1./\thisrow{N}, y expr=\thisrow{eta}/\thisrow{err}]%
{figures/data/high/P1/curve_16.3.txt};

\plot%
table[x expr = 1./\thisrow{N}, y expr=\thisrow{eta}/\thisrow{err}]%
{figures/data/high/P1/curve_32.3.txt}%
node[pos=0.8,pin=-45:{$\ell = 32$}] {};

\end{axis}
\end{tikzpicture}
\end{minipage}

\begin{minipage}{.45\linewidth}
\begin{tikzpicture}
\begin{axis}%
[%
	width  = \linewidth,
	xlabel = {$h$ $\;$ ($p = 3$)},
	ylabel = {$\eta/\norm{\BE-\BE_h}_{\ccurl,\omega,\Omega}$},
	xmode = log,
	ymode = log,
	x dir  = reverse,
	ymax = 50,
	ymin = 0.1666666,
	xmax = 1,
	xmin = 1/4096
]

\plot%
table[x expr = 1./\thisrow{N}, y expr=\thisrow{eta}/\thisrow{err}]%
{figures/data/high/P2/curve_1.3.txt}
node[pos=0.15,pin={[pin distance=0.5cm]90:{$\ell = 1$}}] {};

\plot%
table[x expr = 1./\thisrow{N}, y expr=\thisrow{eta}/\thisrow{err}]%
{figures/data/high/P2/curve_2.3.txt};

\plot%
table[x expr = 1./\thisrow{N}, y expr=\thisrow{eta}/\thisrow{err}]%
{figures/data/high/P2/curve_4.3.txt};

\plot%
table[x expr = 1./\thisrow{N}, y expr=\thisrow{eta}/\thisrow{err}]%
{figures/data/high/P2/curve_8.3.txt};

\plot%
table[x expr = 1./\thisrow{N}, y expr=\thisrow{eta}/\thisrow{err}]%
{figures/data/high/P2/curve_16.3.txt};

\plot%
table[x expr = 1./\thisrow{N}, y expr=\thisrow{eta}/\thisrow{err}]%
{figures/data/high/P2/curve_32.3.txt}
node[pos=0.8,pin=-45:{$\ell = 32$}] {};

\end{axis}
\end{tikzpicture}
\end{minipage}
\begin{minipage}{.45\linewidth}
\begin{tikzpicture}
\begin{axis}%
[%
	width  = \linewidth,
	xlabel = {$h$ $\;$ ($p = 4$)},
	ylabel = {$\eta/\norm{\BE-\BE_h}_{\ccurl,\omega,\Omega}$},
	xmode = log,
	ymode = log,
	x dir  = reverse,
	ymax = 50,
	ymin = 0.1666666,
	xmax = 1,
	xmin = 1/4096
]

\plot%
table[x expr = 1./\thisrow{N}, y expr=\thisrow{eta}/\thisrow{err}]%
{figures/data/high/P3/curve_1.3.txt}
node[pos=0.15,pin={[pin distance=0.5cm]90:{$\ell = 1$}}] {};

\plot%
table[x expr = 1./\thisrow{N}, y expr=\thisrow{eta}/\thisrow{err}]%
{figures/data/high/P3/curve_2.3.txt};

\plot%
table[x expr = 1./\thisrow{N}, y expr=\thisrow{eta}/\thisrow{err}]%
{figures/data/high/P3/curve_4.3.txt};

\plot%
table[x expr = 1./\thisrow{N}, y expr=\thisrow{eta}/\thisrow{err}]%
{figures/data/high/P3/curve_8.3.txt};

\plot%
table[x expr = 1./\thisrow{N}, y expr=\thisrow{eta}/\thisrow{err}]%
{figures/data/high/P3/curve_16.3.txt};

\plot%
table[x expr = 1./\thisrow{N}, y expr=\thisrow{eta}/\thisrow{err}]%
{figures/data/high/P3/curve_32.3.txt}
node[pos=0.8,pin=-45:{$\ell = 32$}] {};

\end{axis}
\end{tikzpicture}
\end{minipage}
\caption{PEC cavity: high frequency example}
\label{figure_pec_cavity_high}
\vspace{-0.25cm}
\end{figure}
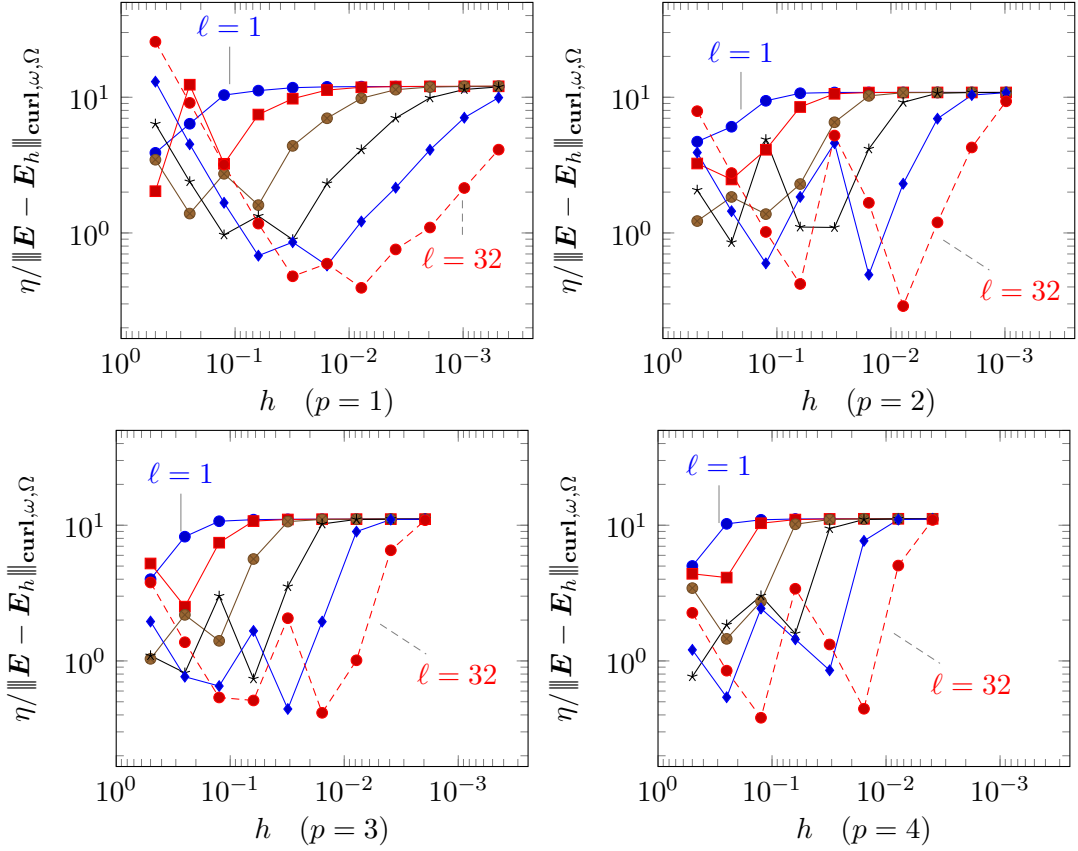

Considering a fixed polynomial degree $p$, in both cases, one sees that when the
frequency gets closer to the resonance value or is increased, the reliability constant
is larger for coarse meshes. Asymptotically, when $h \to 0$ as the mesh is refined,
the same effectivity index is achieved for all frequencies. This  highlights our key
theoretical finding, stating that the reliability and efficiency constants
are independent of the frequency if the mesh is sufficiently refined.

We also observe that the ``asymptotic'' effectivity index is achieved faster for higher values
of $p$. This is in perfect accordance with our analysis
(see in particular \eqref{eq_estimate_gbaE_simplified}),
since the approximation factor decreases when $p$ is increased for a fixed mesh.

\subsection{Analytical solution in a PML medium}
\label{experiment_pml}

For this experiment, we still consider the square $\Omega_0 \eq (-1,1)^2$,
that we surround with a PML of thickness $\ell \eq 1/4$. Hence, the
domain is $\Omega \eq (-1-\ell,1+\ell)^2$. Following
\cite{berenger_1994,berenger_1996,monk_2003a}, the coefficients are defined as
\begin{equation*}
\ee = \left (
\begin{array}{cc}
d_2/d_1 & 0
\\
0       & d_1/d_2
\end{array}
\right ),
\qquad
\mu = d_1 d_2,
\end{equation*}
where $d_j \eq d_j(\bx_j) \eq 1 + \sigma/(i\omega)\mathbf 1_{\bx_j > 1}$ with
$\sigma = (3/4)\omega$. Notice that $\ee = \BTI$ and $\mu = 1$ inside $\Omega_0$
and are only modified inside the PML.

We consider a plane wave travelling in the direction $\bd \eq (\cos \phi, \sin \phi)$ and
polarized along $\bp \eq (\sin \phi, -\cos \phi)$, with $\phi = \pi/12$.
We thus let $\boldsymbol \xi_\phi(\bx) \eq \bp e^{-i\omega \bd \cdot \bx}$.
Then, we separate the ``total'' and ``scattered'' field regions with
a cutoff function $\chi(\bx) \eq \widehat \chi(|\bx|)$, where
$\widehat \chi$ is the unique element of $C^2(\mathbb R)$ such that
$\widehat \chi(t) = 1$ if $t < 0.8$, $\widehat \chi(t) = 0$ if $t > 0.9$,
and $\widehat \chi|_{[0.8,0.9]}$ is a polynomial of degree 5. All in all,
the solution reads $\BE \eq \chi \boldsymbol \xi_{\phi}$. The corresponding
right-hand side, $\BJ \eq -\omega^2 \BE + \curl \curl \BE$, is supported in the ring
$0.8 \leq |\bx| \leq 0.9$.

\begin{figure}
\begin{minipage}{.45\linewidth}
\begin{tikzpicture}
\begin{axis}%
[%
	width  = \linewidth,
	xlabel = {$h$ $\;$ ($p = 1$)},
	ylabel = {$\eta/\norm{\BE-\BE_h}_{\ccurl,\omega,\Omega}$},
	xmode = log,
	ymode = log,
	x dir  = reverse,
	ymax = 50,
	ymin = 0.33333333333,
	xmax = 1/8,
	xmin = 1/2048
]

\plot%
table[x expr = .8/\thisrow{N}, y expr=\thisrow{eta}/\thisrow{err}]%
{figures/data/pml_pw/P0/curve_1.00.txt}
node[pos=0.1,pin={[pin distance=.25cm]90:{$\omega = 2\pi$}}] {};

\plot%
table[x expr = .8/\thisrow{N}, y expr=\thisrow{eta}/\thisrow{err}]%
{figures/data/pml_pw/P0/curve_2.00.txt};

\plot%
table[x expr = .8/\thisrow{N}, y expr=\thisrow{eta}/\thisrow{err}]%
{figures/data/pml_pw/P0/curve_4.00.txt};

\plot%
table[x expr = .8/\thisrow{N}, y expr=\thisrow{eta}/\thisrow{err}]%
{figures/data/pml_pw/P0/curve_8.00.txt};

\plot%
table[x expr = .8/\thisrow{N}, y expr=\thisrow{eta}/\thisrow{err}]%
{figures/data/pml_pw/P0/curve_16.00.txt};

\plot%
table[x expr = .8/\thisrow{N}, y expr=\thisrow{eta}/\thisrow{err}]%
{figures/data/pml_pw/P0/curve_32.00.txt}
node[pos=0.9,pin={[pin distance=.5cm]-90:{$\omega = 64\pi$}}] {};

\end{axis}
\end{tikzpicture}
\end{minipage}
\begin{minipage}{.45\linewidth}
\begin{tikzpicture}
\begin{axis}%
[%
	width  = \linewidth,
	xlabel = {$h$ $\;$ ($p=2$)},
	ylabel = {$\eta/\norm{\BE-\BE_h}_{\ccurl,\omega,\Omega}$},
	xmode = log,
	ymode = log,
	x dir  = reverse,
	ymax = 50,
	ymin = 0.33333333333,
	xmax = 1/8,
	xmin = 1/2048
]

\plot%
table[x expr = .8/\thisrow{N}, y expr=\thisrow{eta}/\thisrow{err}]%
{figures/data/pml_pw/P1/curve_1.00.txt}
node[pos=0.1,pin={[pin distance=.125cm]90:{$\omega = 2\pi$}}] {};

\plot%
table[x expr = .8/\thisrow{N}, y expr=\thisrow{eta}/\thisrow{err}]%
{figures/data/pml_pw/P1/curve_2.00.txt};

\plot%
table[x expr = .8/\thisrow{N}, y expr=\thisrow{eta}/\thisrow{err}]%
{figures/data/pml_pw/P1/curve_4.00.txt};

\plot%
table[x expr = .8/\thisrow{N}, y expr=\thisrow{eta}/\thisrow{err}]%
{figures/data/pml_pw/P1/curve_8.00.txt};

\plot%
table[x expr = .8/\thisrow{N}, y expr=\thisrow{eta}/\thisrow{err}]%
{figures/data/pml_pw/P1/curve_16.00.txt};

\plot%
table[x expr = .8/\thisrow{N}, y expr=\thisrow{eta}/\thisrow{err}]%
{figures/data/pml_pw/P1/curve_32.00.txt}
node[pos=0.9,pin={[pin distance=1cm]-90:{$\omega = 64\pi$}}] {};

\end{axis}
\end{tikzpicture}
\end{minipage}

\begin{minipage}{.45\linewidth}
\begin{tikzpicture}
\begin{axis}%
[%
	width  = \linewidth,
	xlabel = {$h$ $\;$ ($p=3$)},
	ylabel = {$\eta/\norm{\BE-\BE_h}_{\ccurl,\omega,\Omega}$},
	xmode = log,
	ymode = log,
	x dir  = reverse,
	ymax = 50,
	ymin = 0.33333333333,
	xmax = 1/8,
	xmin = 1/2048
]

\plot%
table[x expr = .8/\thisrow{N}, y expr=\thisrow{eta}/\thisrow{err}]%
{figures/data/pml_pw/P2/curve_1.00.txt}
node[pos=0.1,pin={[pin distance=0.5cm]90:{$\omega = 2\pi$}}] {};

\plot%
table[x expr = .8/\thisrow{N}, y expr=\thisrow{eta}/\thisrow{err}]%
{figures/data/pml_pw/P2/curve_2.00.txt};

\plot%
table[x expr = .8/\thisrow{N}, y expr=\thisrow{eta}/\thisrow{err}]%
{figures/data/pml_pw/P2/curve_4.00.txt};

\plot%
table[x expr = .8/\thisrow{N}, y expr=\thisrow{eta}/\thisrow{err}]%
{figures/data/pml_pw/P2/curve_8.00.txt};

\plot%
table[x expr = .8/\thisrow{N}, y expr=\thisrow{eta}/\thisrow{err}]%
{figures/data/pml_pw/P2/curve_16.00.txt};

\plot%
table[x expr = .8/\thisrow{N}, y expr=\thisrow{eta}/\thisrow{err}]%
{figures/data/pml_pw/P2/curve_32.00.txt}
node[pos=0.6,pin=-45:{$\omega = 64\pi$}] {};

\end{axis}
\end{tikzpicture}
\end{minipage}
\begin{minipage}{.45\linewidth}
\begin{tikzpicture}
\begin{axis}%
[%
	width  = \linewidth,
	xlabel = {$h$ $\;$ ($p=4$)},
	ylabel = {$\eta/\norm{\BE-\BE_h}_{\ccurl,\omega,\Omega}$},
	xmode = log, ymode = log,
	x dir  = reverse,
	ymax = 50,
	ymin = 0.33333333333,
	xmax = 1/8,
	xmin = 1/2048
]

\plot%
table[x expr = .8/\thisrow{N}, y expr=\thisrow{eta}/\thisrow{err}]%
{figures/data/pml_pw/P3/curve_1.00.txt}
node[pos=0.1,pin={[pin distance=.5cm]90:{$\omega = 2\pi$}}] {};

\plot%
table[x expr = .8/\thisrow{N}, y expr=\thisrow{eta}/\thisrow{err}]%
{figures/data/pml_pw/P3/curve_2.00.txt};

\plot%
table[x expr = .8/\thisrow{N}, y expr=\thisrow{eta}/\thisrow{err}]%
{figures/data/pml_pw/P3/curve_4.00.txt};

\plot%
table[x expr = .8/\thisrow{N}, y expr=\thisrow{eta}/\thisrow{err}]%
{figures/data/pml_pw/P3/curve_8.00.txt};

\plot%
table[x expr = .8/\thisrow{N}, y expr=\thisrow{eta}/\thisrow{err}]%
{figures/data/pml_pw/P3/curve_16.00.txt};

\plot%
table[x expr = .8/\thisrow{N}, y expr=\thisrow{eta}/\thisrow{err}]%
{figures/data/pml_pw/P3/curve_32.00.txt}
node[pos=0.8,pin={[pin distance=.5cm]-45:{$\omega = 64\pi$}}] {};

\end{axis}
\end{tikzpicture}
\end{minipage}
\caption{PML medium: plane wave in free space}
\label{figure_pml_analytical_high}
\vspace{-0.75cm}
\end{figure}
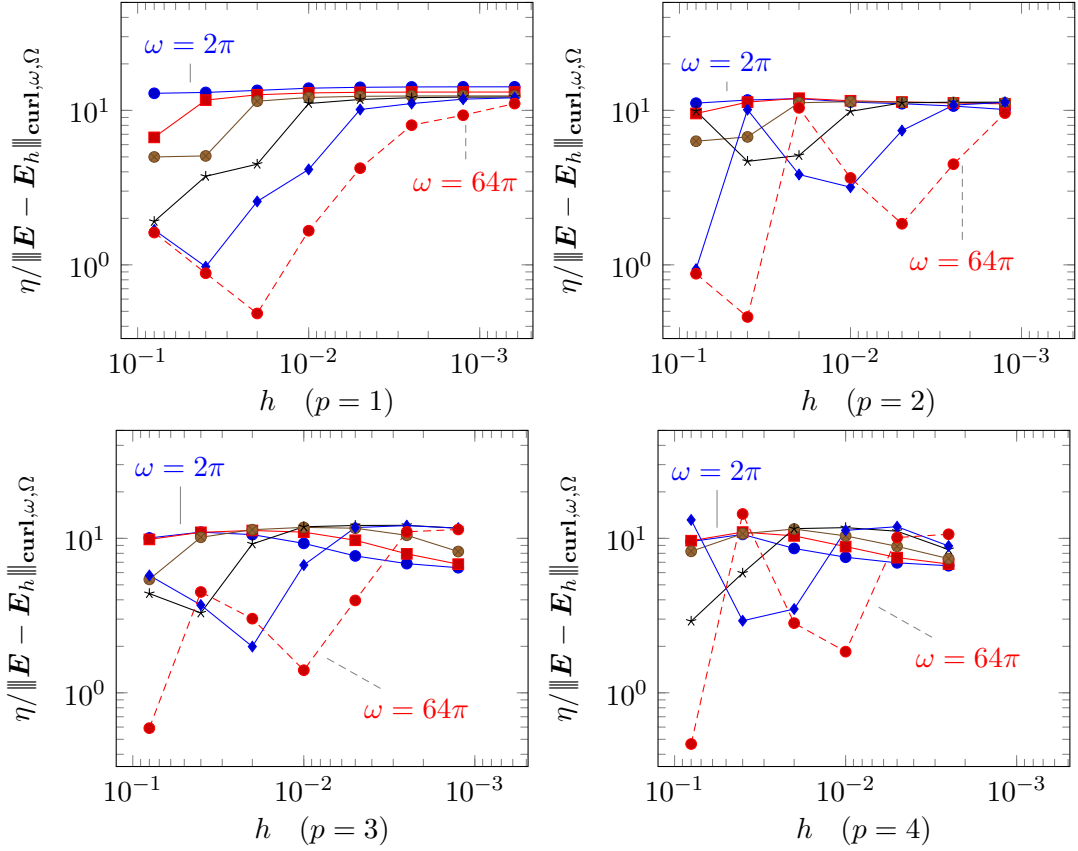

The effectivity indices obtained for different frequencies $\omega$, mesh sizes $h$
and polynomial degrees $p$ are plotted on Figure \ref{figure_pml_analytical_high}. The
conclusions are similar to the previous experiment and perfectly illustrate our analysis.
In particular, for a fixed $p$, the asymptotic effectivity index is independent of
the frequency $\omega$. We observe a preasymptotic range where the effectivity index is higher,
which corresponds to a large approximation factor $\gbaE$. As expected,
the asymptotic regime is achieved faster for higher values of $p$.

\subsection{Scattering by a penetrable obstacle}
\label{experiment_scattering}

We consider scattering by the penetrable obstacle $G \eq (-1/4,1/4)^2$. We select
$\Omega_0$, the PMLs, and $\BJ$ as in Experiment \ref{experiment_pml}, we also keep the definition
of $\ee$ and $\mu$ in $\Omega \setminus G$, but set
\begin{equation*}
\mu \eq \frac{1}{4},
\qquad
\ee \eq \left (
\begin{array}{cc}
8 &  0
\\
0 & 32
\end{array}
\right )
\end{equation*}
in $G$. Here, the analytical solution is unavailable. Given $\BE_h$, we then
compute errors compared to $\widetilde \BE$, where $\widetilde \BE$ is computed on the
same mesh as $\BE_h$ with $p = 7$. Figure \ref{figure_image_scattering} presents
the most accurate approximation of the solution computed for different frequencies.

\begin{figure}
\input{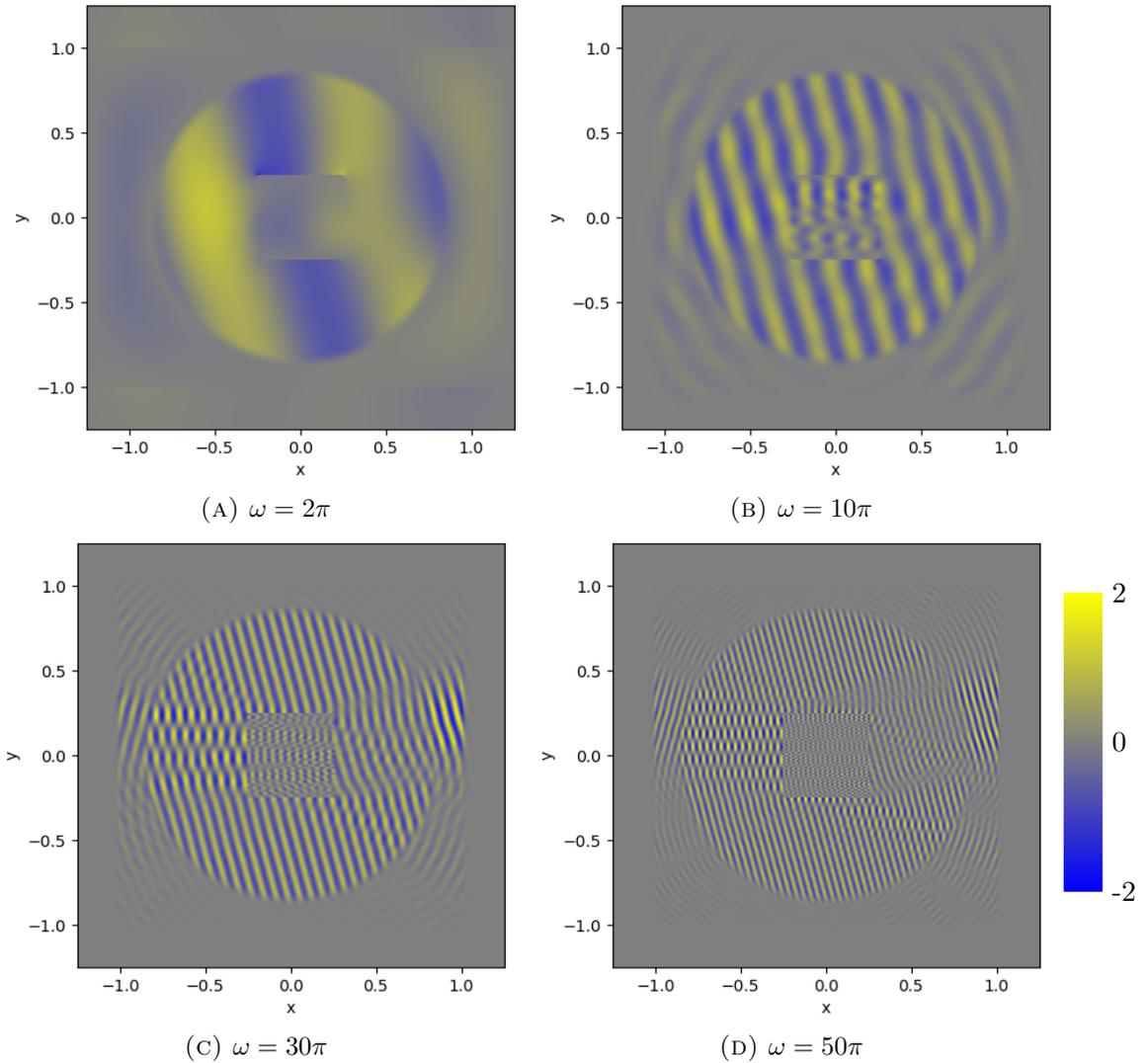}
\vspace{-0.25cm}
\caption{$\Re \widetilde \BE_2$ computed at the last iteration of the adaptive algorithm}
\label{figure_image_scattering}
\vspace{-0.75cm}
\end{figure}

The goal of this experiment is to analyze the ability of the proposed estimator
to drive an adaptive mesh refinement algorithm.  In contrast to the previous experiments, we
consider general unstructured meshes that are generated using the {\tt MMG}
software package \cite{mmg3d}. This package takes as input an already existing mesh
and a set of maximal mesh sizes associated with each vertex of the input mesh.
The output is a new mesh, locally refined so that the prescribed mesh
sizes are respected. We use the {\tt MMG} package, together with D\"orfler's marking
\cite{dorfler_1996a}, to iteratively refine the mesh.

\begin{algo}[Adaptive refinements]
\label{algo_hp}
Specifically, our adaptive algorithm is as follows:
\begin{itemize}
\item
Given a mesh $\CT_h$, compute the associated discrete solution and estimators $\eta_K$.
\item
Order the elements $K$ by decreasing values of $\eta_K$ and constitute a set $\CM \subset \CT_h$
by adding the elements in the list until \begin{equation*}
\sum_{K \in \CM} \eta_K^2 \geq \theta \sum_{K \in \CT_h} \eta_K^2
\end{equation*}
where $\theta \eq 0.1$.
\item
Associate with each element $K \in \CT_h$ a ``desired size'' $\widetilde h_K$.
This is done by setting $\widetilde h_K \eq h_K/2$ if $K \in \CM$, and
$\widetilde h_K \eq h_K$ otherwise.
\item
Associate with each vertex $\ba \in \CV_h$ a ``desired size'' $\widetilde h_{\ba}$
that is the minimum of the desired size of the elements $K \in \CT_h$ having $\ba$ as a vertex.
\item Use {\tt MMG} with $\CT_h$ and $\{\widetilde h_{\ba}\}_{\ba \in \CV_h}$ to produce
a new mesh $\widetilde \CT_h$.
\item Perform a new iteration with $\CT_h \eq \widetilde \CT_h$.
\end{itemize}
\end{algo}

On Figure \ref{figure_mesh}, we present the initial mesh, as well as the meshes
obtained after $10$ iterations of Algorithm \ref{algo_hp} for different values
of $p$ and $\omega$. In all cases, the regions selected for refinements are
understandable. On the top-right panel, the frequency is rather low, so that
the cutoff function employed to inject the incident wave is the main source
of fast oscillations. This area is clearly refined. Also, one clearly sees
that the mesh is strongly refined in the vicinity of the corners of the scatterer,
which is to be expected, due to singularities \cite{bonito_guermond_luddens_2013a,costabel_dauge_nicaise_1999a}.
In the two bottom panels, we observe that the mesh is essentially refined inside the scatterer
and coarser inside the PML. On the one hand, it is perfectly suited for the mesh to be refined
inside the scatterer since the wavespeed is smaller (hence, the solution
is more oscillatory) in this region. On the other hand, it is understandable that
the mesh is coarse in the PML, as outgoing radiations are rapidly absorbed.

\begin{figure}
\input{figures/mesh}
\vspace{-0.25cm}
\caption{Initial mesh and meshes obtained at iteration 10 of the algorithm
for different polynomial degrees and frequencies in Experiment
\ref{experiment_scattering}}
\label{figure_mesh}
\vspace{-0.75cm}
\end{figure}

Figure \ref{figure_scattering} provides more quantitative results.
On the one hand, we plot the relative error
\begin{equation*}
\text{Error}
\eq
100\ \frac{\enorm{\widetilde \BE-\BE_h}}{\enorm{\widetilde \BE}}
\end{equation*}
measured in percentage against the number of degrees of freedom $N$
at each iteration. We see that asymptotically, the error behaves as
$O(N^{-p/2})$, which is the optimal rate. This indicates that the
produced meshes are optimal, and thus, adequately refined for all the
frequencies and polynomial degrees considered. On the other hand, we also
plot the effectivity index against the iteration number. The curves we
obtain are similar to those presented above for uniform meshes,
comforting our key theoretical findings. Specifically,
the error is underestimated on coarse meshes, and this underestimation
is more pronounced for higher frequencies. Asymptotically, however, 
the effectivity index becomes independent of the frequency.

{\color{black}
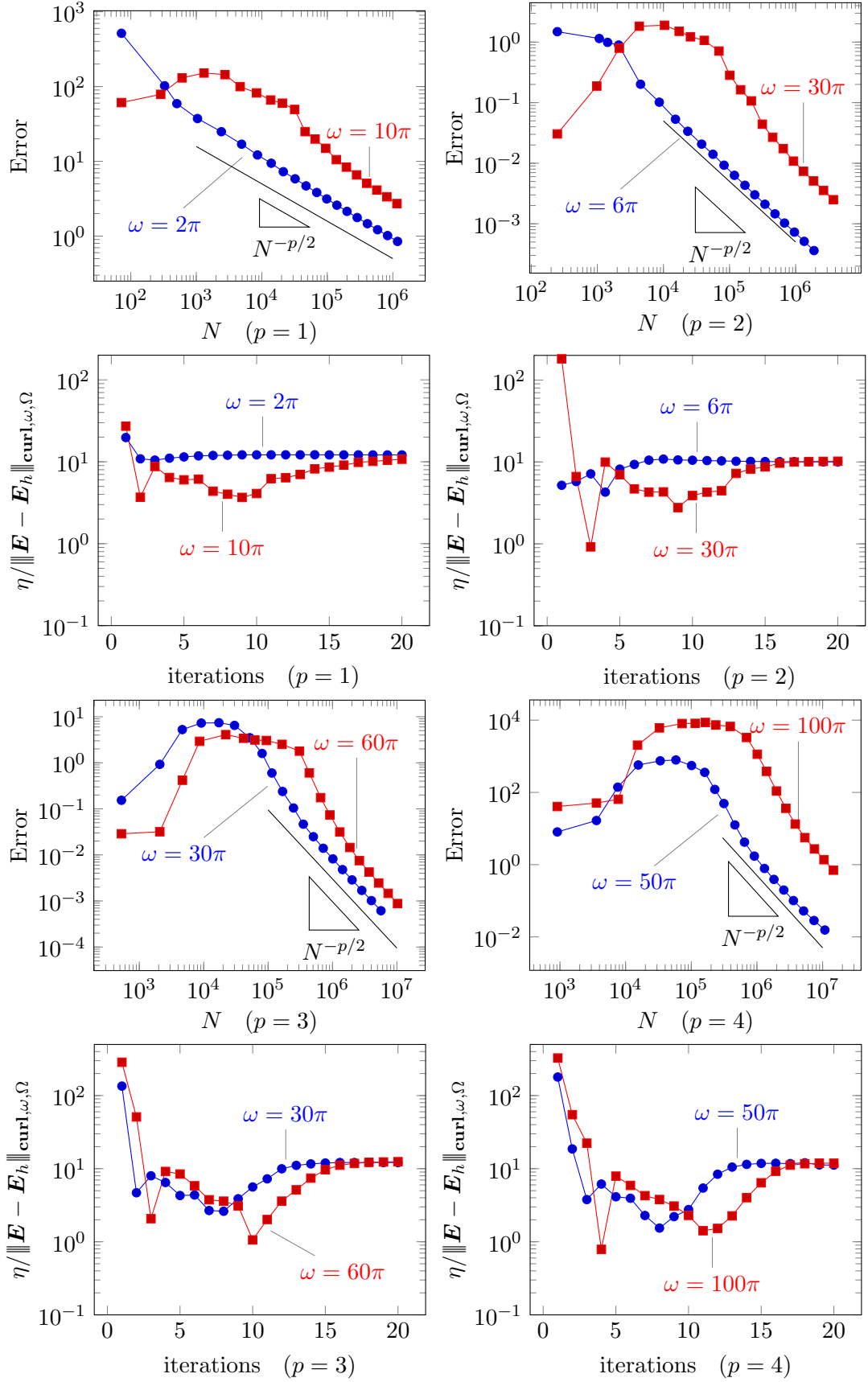
\begin{figure}
\begin{minipage}{.45\linewidth}
\begin{tikzpicture}
\begin{axis}%
[%
	width  = \linewidth,
	xlabel = {$N$ $\;$ ($p = 1$)},
	ylabel = {Error},
	xmode = log,
	ymode = log,
]

\plot%
table[x = nr_dofs, y expr=100*\thisrow{true_err}/15.8379]%
{figures/data/scattering/P0_F1/errors.txt}
node[pos=0.5,pin={[pin distance=1cm]-125:{$\omega = 2\pi$}}] {};

\plot%
table[x = nr_dofs, y expr=100*\thisrow{true_err}/74.23]%
{figures/data/scattering/P0_F5/errors.txt}
node[pos=0.9,pin={90:{$\omega = 10\pi$}}] {};

\plot[black,solid,domain=1e3:1e6] {5e2*x^(-0.5)};

\SlopeTriangle{.5}{-.15}{.2}{-.5}{$N^{-p/2}$}{}

\end{axis}
\end{tikzpicture}
\end{minipage}
\begin{minipage}{.45\linewidth}
\begin{tikzpicture}
\begin{axis}%
[%
	width  = \linewidth,
	xlabel = {$N$ $\;$ ($p = 2$)},
	ylabel = {Error},
	xmode = log,
	ymode = log,
]

\plot%
table[x = nr_dofs, y expr=\thisrow{true_err}/44.9593]%
{figures/data/scattering/P1_F3/errors.txt}
node[pos=0.5,pin={[pin distance=1cm]-125:{$\omega = 6\pi$}}] {};

\plot%
table[x = nr_dofs, y expr=\thisrow{true_err}/220.788]%
{figures/data/scattering/P1_F15/errors.txt}
node[pos=0.9,pin={[pin distance=1cm]90:{$\omega = 30\pi$}}] {};

\plot[black,solid,domain=1e4:1e6] {5e2*x^(-1)};

\SlopeTriangle{.5}{-.15}{.15}{-1}{$N^{-p/2}$}{}

\end{axis}
\end{tikzpicture}
\end{minipage}

\begin{minipage}{.45\linewidth}
\begin{tikzpicture}
\begin{axis}%
[%
	width  = \linewidth,
	xlabel = {iterations $\;$ ($p = 1$)},
	ylabel = {$\eta/\norm{\BE-\BE_h}_{\ccurl,\omega,\Omega}$},
	ymode = log,
	ymin = .1,
	ymax = 200,
]

\plot%
table[x = iter, y expr=\thisrow{eta}/\thisrow{true_err}]%
{figures/data/scattering/P0_F1/errors.txt}
node[pos=0.5,pin=90:{$\omega = 2\pi$}] {};

\plot%
table[x = iter, y expr=\thisrow{eta}/\thisrow{true_err}]%
{figures/data/scattering/P0_F5/errors.txt}
node[pos=0.4,pin=-90:{$\omega = 10\pi$}] {};

\end{axis}
\end{tikzpicture}
\end{minipage}
\begin{minipage}{.45\linewidth}
\begin{tikzpicture}
\begin{axis}%
[%
	width  = \linewidth,
	xlabel = {iterations $\;$ ($p = 2$)},
	ylabel = {$\eta/\norm{\BE-\BE_h}_{\ccurl,\omega,\Omega}$},
	ymode = log,
	ymin = .1,
	ymax = 200,
]

\plot%
table[x = iter, y expr=\thisrow{eta}/\thisrow{true_err}]%
{figures/data/scattering/P1_F3/errors.txt}
node[pos=0.5,pin=90:{$\omega = 6\pi$}] {};

\plot%
table[x = iter, y expr=\thisrow{eta}/\thisrow{true_err}]%
{figures/data/scattering/P1_F15/errors.txt}
node[pos=0.6,pin=-90:{$\omega = 30\pi$}] {};

\end{axis}
\end{tikzpicture}
\end{minipage}

\begin{minipage}{.45\linewidth}
\begin{tikzpicture}
\begin{axis}%
[%
	width  = \linewidth,
	xlabel = {$N$ $\;$ ($p = 3$)},
	ylabel = {Error},
	xmode = log,
	ymode = log,
]

\plot%
table[x = nr_dofs, y expr=\thisrow{true_err}/44.9593]%
{figures/data/scattering/P2_F15/errors.txt}
node[pos=0.55,pin={[pin distance=1cm]-125:{$\omega = 30\pi$}}] {};

\plot%
table[x = nr_dofs, y expr=\thisrow{true_err}/220.788]%
{figures/data/scattering/P2_F30/errors.txt}
node[pos=0.85,pin={[pin distance=1.5cm]90:{$\omega = 60\pi$}}] {};

\plot[black,solid,domain=1e5:1e7] {3e6*x^(-1.5)};

\SlopeTriangle{.65}{-.15}{.15}{-1.5}{$N^{-p/2}$}{}

\end{axis}
\end{tikzpicture}
\end{minipage}
\begin{minipage}{.45\linewidth}
\begin{tikzpicture}
\begin{axis}%
[%
	width  = \linewidth,
	xlabel = {$N$ $\;$ ($p = 4$)},
	ylabel = {Error},
	xmode = log,
	ymode = log,
]

\plot%
table[x = nr_dofs, y expr=100*\thisrow{true_err}/74.23]%
{figures/data/scattering/P3_F25/errors.txt}
node[pos=0.55,pin={[pin distance=1cm]-125:{$\omega = 50\pi$}}] {};

\plot%
table[x = nr_dofs, y expr=100*\thisrow{true_err}/15.8379]%
{figures/data/scattering/P3_F50/errors.txt}
node[pos=0.85,pin={[pin distance=1.3cm]90:{$\omega = 100\pi$}}] {};

\plot[black,solid,domain=3e5:1e7] {5e11*x^(-2)};

\SlopeTriangle{.6}{-.15}{.2}{-2.}{$N^{-p/2}$}{}

\end{axis}
\end{tikzpicture}
\end{minipage}

\begin{minipage}{.45\linewidth}
\begin{tikzpicture}
\begin{axis}%
[%
	width  = \linewidth,
	xlabel = {iterations $\;$ ($p = 3$)},
	ylabel = {$\eta/\norm{\BE-\BE_h}_{\ccurl,\omega,\Omega}$},
	ymode = log,
	ymin = .1,
	ymax = 500,
]

\plot%
table[x = iter, y expr=\thisrow{eta}/\thisrow{true_err}]%
{figures/data/scattering/P2_F15/errors.txt}
node[pos=0.65,pin=90:{$\omega = 30\pi$}] {};

\plot%
table[x = iter, y expr=\thisrow{eta}/\thisrow{true_err}]%
{figures/data/scattering/P2_F30/errors.txt}
node[pos=0.6,pin=-45:{$\omega = 60\pi$}] {};

\end{axis}
\end{tikzpicture}
\end{minipage}
\begin{minipage}{.45\linewidth}
\begin{tikzpicture}
\begin{axis}%
[%
	width  = \linewidth,
	xlabel = {iterations $\;$ ($p = 4$)},
	ylabel = {$\eta/\norm{\BE-\BE_h}_{\ccurl,\omega,\Omega}$},
	ymode = log,
	ymin = .1,
	ymax = 500,
]

\plot%
table[x = iter, y expr=\thisrow{eta}/\thisrow{true_err}]%
{figures/data/scattering/P3_F25/errors.txt}
node[pos=0.7,pin=90:{$\omega = 50\pi$}] {};

\plot%
table[x = iter, y expr=\thisrow{eta}/\thisrow{true_err}]%
{figures/data/scattering/P3_F50/errors.txt}
node[pos=0.65,pin=-90:{$\omega = 100\pi$}] {};

\end{axis}
\end{tikzpicture}
\end{minipage}
\caption{Scattering by a penetrable obstacle}
\label{figure_scattering}
\end{figure}
}


\section{Conclusion}
\label{sec_conclusion}

We analyzed residual-based {\it a posteriori} error estimators for the discretization
of time-harmonic Maxwell's equations in heterogeneous media. We have focused on
(conforming) N\'ed\'elec finite element discretizations of the second-order
formulation. The novelty of our work is that we derive frequency-explicit
reliability and efficiency estimates.

Our findings generalize previous results established for scalar
wave propagation problems modeled by the Helmholtz equation. Specifically,
we establish that the efficiency constant is independent of the frequency as soon as
the number of degrees of freedom per wavelength is bounded below.
On the other hand, we show that the reliability constant is bounded independently of the frequency
for sufficiently refined meshes, but can become large at high frequencies
(or close to resonances) for coarse meshes and/or low polynomial degrees.

We presented numerical experiments that highlight these key theoretical features.
We produced three different test cases, including interior problems as well
as scattering problems with perfectly matched layers. In all cases, the behavior
of the estimator fits the theoretical predictions. Finally, the estimator has been
employed to drive an adaptive mesh refinement algorithm based on D\"orfler's marking. We
obtained optimal convergence rates in terms of the number of degrees of freedom,
indicating that the proposed estimator is perfectly suited for adaptivity
purposes.

\bibliographystyle{amsplain}
\bibliography{bibliography.bib}

\end{document}